\documentclass[sts,
	reqno,
]{imsart}

\usepackage{amsthm, amssymb, amsmath}
\usepackage{bbm}
\makeatletter
\let\c@author\relax
\makeatother
\usepackage[backend=biber, citestyle=authoryear-comp, maxcitenames=3, style=authoryear, uniquelist=false]{biblatex}
\usepackage{mathrsfs}
\usepackage{mathtools}
\usepackage{euscript}
\usepackage{color}
\usepackage{tikz}
\usetikzlibrary{arrows}
\usepackage[utf8]{inputenc}
\usepackage[shortlabels]{enumitem}
\usepackage{url}
\usepackage{etex}
\usepackage{alltt}
\usepackage{memhfixc}
\usepackage{bbold}%
\usepackage{cases}
\definecolor{labelkey}{rgb}{0.0, 0.8, 0.3}
\usepackage[top=2.5cm,bottom=2.5cm,left=2.5cm,right=2.5cm,includeheadfoot,asymmetric]{geometry}
\usepackage{algorithm}
\usepackage{algpseudocode}
\usepackage{xfrac}
\usepackage{wrapfig}

\usepackage{custom}

\numberwithin{equation}{section}
\usepackage{hyperref}
\hypersetup{
  colorlinks = true,
  urlcolor = blue, 
  linkcolor = blue,
  citecolor = red}
  
\DeclareNameAlias{sortname}{family-given}

\declaretheorem[name=Corollary]{cor}
\declaretheorem[name=Lemma]{lem}
\declaretheorem[name=Proposition]{prop}
\declaretheorem[name=Remark, style=remark]{rmk}
\declaretheorem[name=Theorem]{thm}
\newtheorem*{thm*}{Theorem}
\definecolor{pink}{cmyk}{0, 1, 0, 0} 
\newcommand\X{X^\star}
\newcommand\dotX{\dot X^\star}
\newcommand\ddotX{\ddot X^\star}
\newcommand\dddotX{\dddot X^\star}
\newcommand\Y{X}
\newcommand{\mus}{\mu^{\star}}

\newcommand{\allcaps}[1]{#1}
\renewcommand{\abstractname}{Abstract}

\newcommand{\placeqed}{\nobreak\enspace$\square$}
\newcommand{\placeqedmm}{\nobreak\enspace\square}

\begin{document}

\begin{frontmatter}

	\title{Fast and Smooth Interpolation on Wasserstein Space}
	\runtitle{Fast and Smooth Interpolation on Wasserstein Space}
	\author{Sinho Chewi \hfill schewi@mit.edu \\
	Julien Clancy \hfill julienc@mit.edu \\
	Thibaut Le Gouic \hfill tlegouic@mit.edu \\
	Philippe Rigollet \hfill rigollet@mit.edu \\
	George Stepaniants \hfill gstepan@mit.edu \\
	Austin J.\ Stromme \hfill astromme@mit.edu}



	\address{{Department of Mathematics} \\
		{Massachusetts Institute of Technology}\\
		{77 Massachusetts Avenue,}\\
		{Cambridge, MA 02139-4307, USA}
	}
	


\runauthor{Clancy et al.}

\begin{abstract}
We propose a new method for smoothly interpolating probability measures using the geometry of optimal transport. To that end, we reduce this problem to the classical Euclidean setting, allowing us to directly leverage the extensive toolbox of spline interpolation. Unlike previous approaches to measure-valued splines, our interpolated curves (i) have a clear interpretation as governing particle flows, which is natural for applications, and (ii) come with the first approximation guarantees on Wasserstein space. Finally, we demonstrate the broad applicability of our interpolation methodology by fitting surfaces of measures using thin-plate splines.
\end{abstract}

\end{frontmatter}

\section{\allcaps{Introduction}}



Smooth interpolation is a fundamental tool in numerical analysis that plays a central role in data science. While this task is traditionally studied on the flat Euclidean space $\R^d$, recent applications have called for interpolation of points living on curved spaces such as smooth manifolds \parencite{noakesCubicSplinesCurved1989} and, more recently, the Wasserstein space of probability measures. An important application arises in single-cell genomic data analysis where the measure $\mus_t$ represents a population of cells at time $t$ of a biological process such as differentiation, and the cells of an organism specialize over the course of early development. In this context, two main questions arise: 
1) to infer the profile of the population at unobserved times; and more importantly 2) to reconstruct the trajectories of individual cells in gene space, that is: given a cell at time $t$, determine its (likely) history and fate. \textcite{HCA} argue that cellular trajectory reconstruction is crucial to unlocking the promises of single-cell genomics. A breakthrough in this direction was recently achieved using optimal transport  by~\textcite{SCHIEBINGER2019928}, but their work does not produce \emph{smooth} trajectories. To illustrate, we display in Figure~\ref{fig:sawtooth} a comparison of their approach with the smooth interpolation methodology developed in the present work. Although we are mainly motivated by cell trajectory reconstruction, we are confident that the flexibility and efficiency of the method will allow it to find applications beyond this scope.

\begin{figure}
\hspace{2em}
\includegraphics[width=0.47\textwidth]{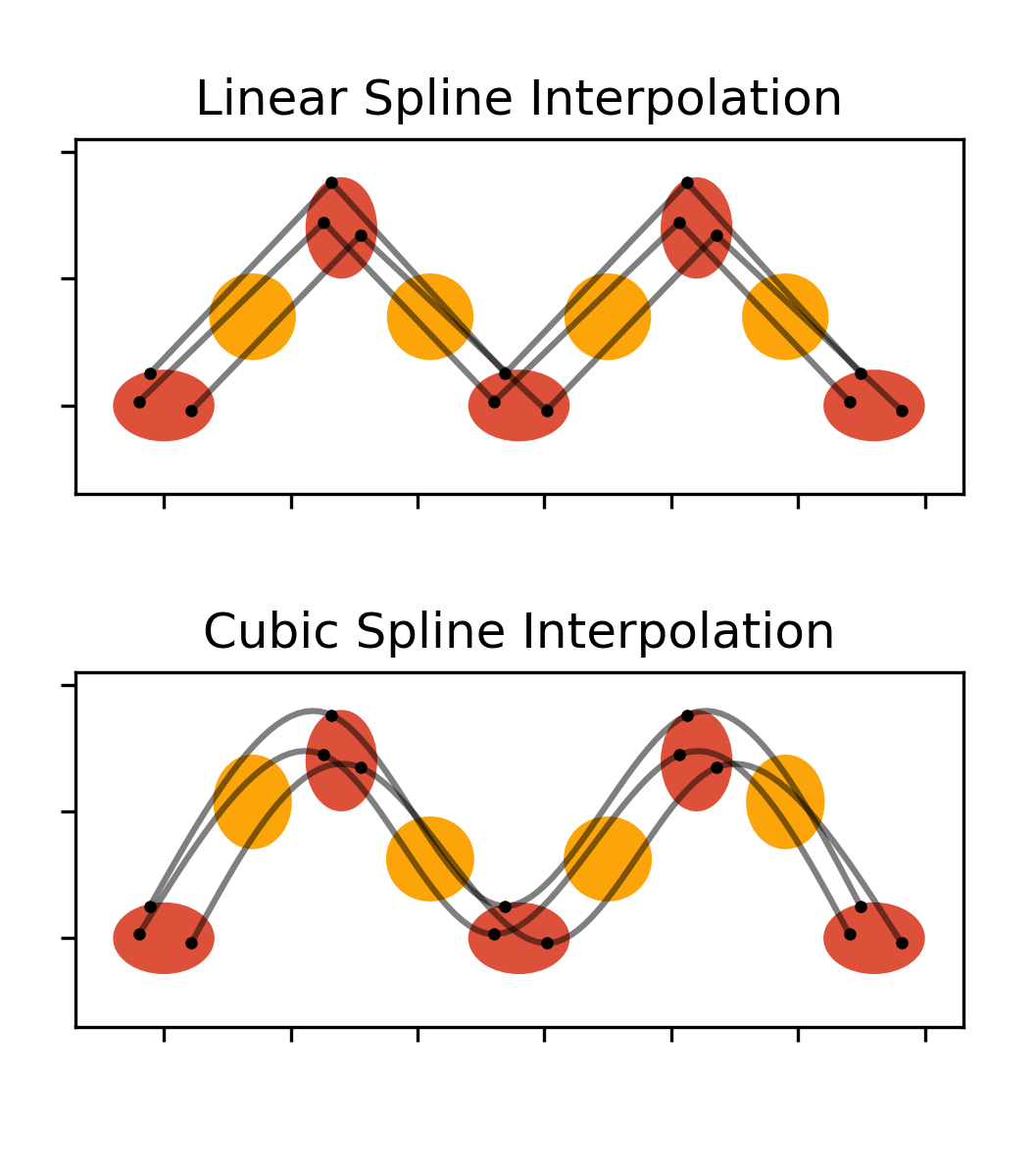}
\vspace{-0.5em}
\caption{\small Piecewise linear and cubic spline interpolation of four Gaussians. The interpolation knots are shown in red and the interpolated Gaussians are shown in orange. See Appendix~\ref{appendix:sawtooth}.}
\label{fig:sawtooth}
\end{figure}
While the first question above is a natural extension of interpolation to the space of probability measures, the second question calls for a specific type of interpolation: one that also reconstructs the (smooth) trajectories of individual particles. Mathematically, the trajectory of a particle (e.g., a cell) is a stochastic process ${(\X_t)}_{t\in [0,1]}$ with smooth sample paths. This leads us to the following problem of trajectory-aware interpolation over the space of probability measures.

\medskip{}
\noindent {\bf The problem.} Let ${(\X_t)}_{t\in [0,1]}$ be a stochastic process on $\R^d$ with $\mc C^2$ sample paths and marginal laws $\X_t\sim \mus_t, t \in [0,1]$. Given $\mus_{t_0},\mus_{t_1},\dotsc,\mus_{t_N}$ at times $0 = t_0 < t_1 < \cdots < t_N = 1$, the task is to construct a stochastic process ${(\Y_t)}_{t\in [0,1]}$ such that $\Y_t$ has  $\mc C^2$ sample paths and the distribution $\mu_t$ of $\Y_t$ \emph{interpolates} the given measures, meaning $\mu_{t_i} = \mus_{t_i}$ for $i=0,1,\dotsc,N$.\\

Throughout, we assume all given measures to be absolutely continuous with finite second moment, and (as advocated in \textcite{SCHIEBINGER2019928}) we equip this space with the $2$-Wasserstein metric $W_2$ and seek an interpolation that reflects this geometry.


\medskip{}
\noindent {\bf Prior work.} This work is at the intersection of interpolation and optimal transport. On the one hand, interpolation in $\R^d$ is very well-developed, with fast and accurate methods ranging from interpolating polynomials and splines to more exotic non-parametric approaches~\parencite{wahbas1990plines}, and with renewed interest due to recent theoretical results~\parencite{Belkin15849}. Our methodology can accommodate all of these options, but we focus on cubic spline interpolations due to their simplicity, theoretical guarantees, and their \emph{curvature-minimizing} property (see Section~\ref{scn:splines}). On the other hand, optimal transport has become a useful tool in the analysis of observations represented in the form of probability measures. Recent computational advances ~\parencite{Cut13,AltWeeRig17,peyre2019computationalot} have led to the development of many methods in statistical optimal transport, from barycenters to geodesic PCA. The present work extends this toolbox by developing a method for smooth interpolation over the Wasserstein space of probability measures.

Splines in Wasserstein space were  considered concurrently and independently  by~\textcite{chenMeasurevaluedSplineCurves2018} and~\textcite{benamouSecondOrderModels2018}. Both papers converge to the same notion of splines, which we call P-splines. Though motivated by particle dynamics, P-splines solve an optimal transport problem that is not guaranteed to have a {\it Monge} solution. Instead, it outputs stochastic processes ${(\Y_t)}_{t\in [0,1]}$ for which $\Y_t$ is not a deterministic function of $\Y_0$. In other words, given an initial position, there is no unique particle trajectory emanating from this position but rather a superposition of such trajectories; see Figure~\ref{fig:pspline_vs_tspline} and the discussion in Section~\ref{scn:splines}. We show that this is not an isolated phenomenon arising from pathological data but applies even to the canonical example of one-dimensional Gaussian distributions. This limitation, together with a relatively heavy computational cost, severely hinders the deployment of P-splines in applications, ours included, especially where interpretation is a priority.

We review these prior works and their motivations in Section~\ref{scn:splines}. We remark however that the algorithm we ultimately propose requires considerably less technical machinery to describe compared to these prior works, and we recommend that readers who simply wish to understand our method skip directly to Section~\ref{scn:transport_spline_alg}.

\begin{figure}
    \centering
    \includegraphics[width=0.475\textwidth]{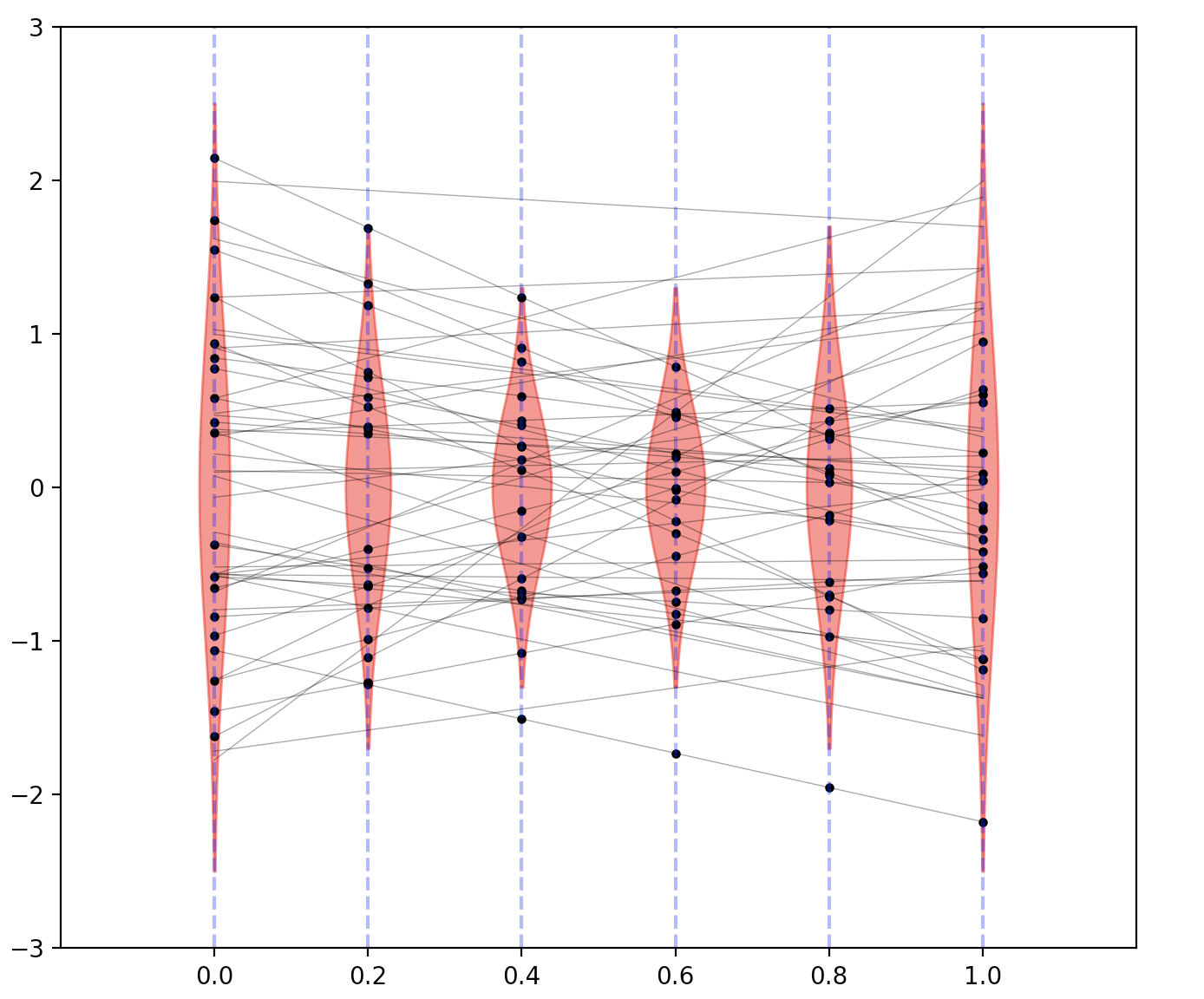}
    \includegraphics[width=0.475\textwidth]{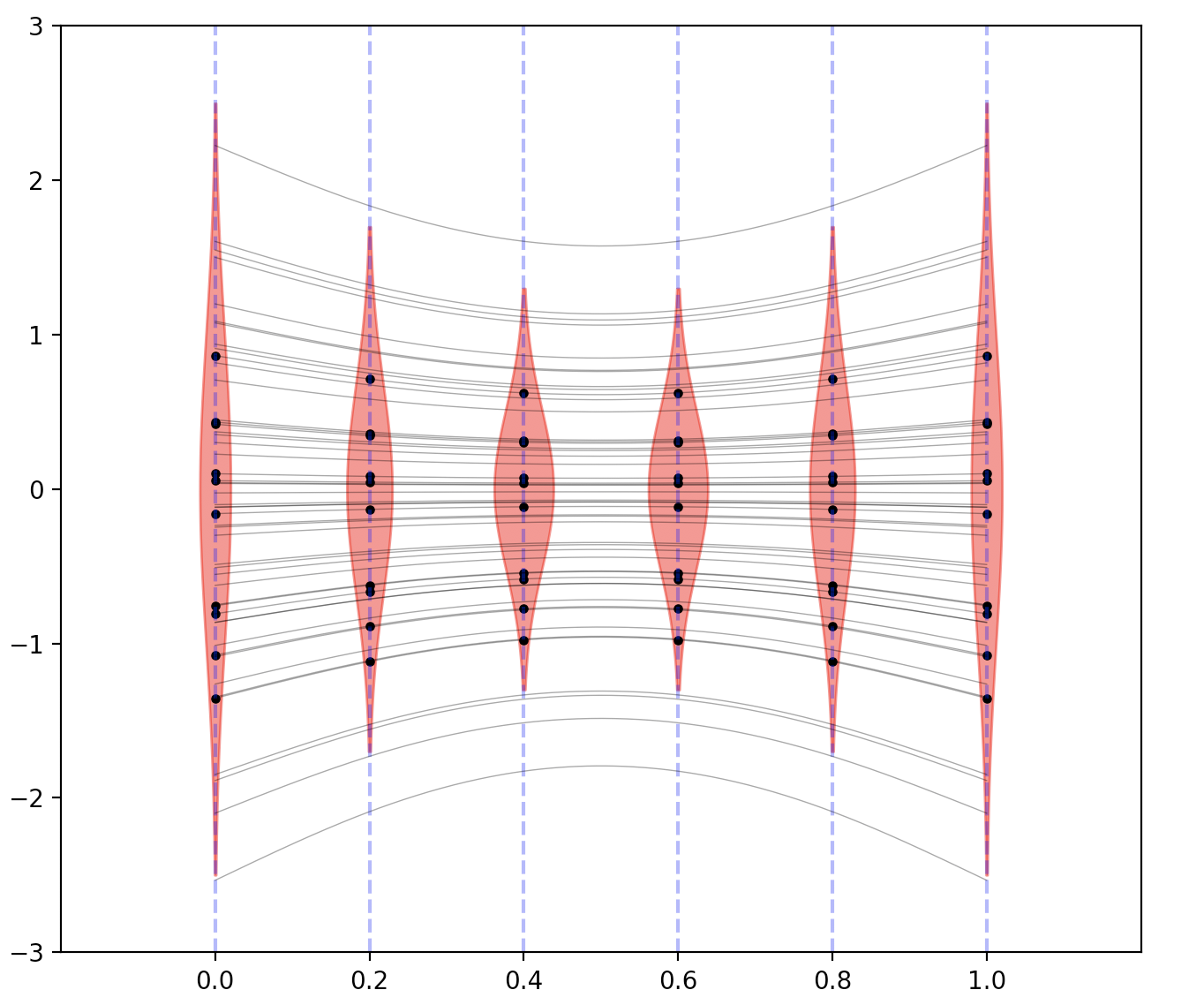}
    \vspace{0.2em}
    \caption{\small A comparison of 50 trajectories sampled from P-splines and transport splines for the Gaussian interpolation problem in Proposition~\ref{prop:non_Monge_p_spline} (see Appendix~\ref{appendix:non_Monge_p_spline} for a detailed discussion). The first figure shows trajectories drawn from the P-spline interpolation, while the second shows trajectories from our method.}
    \label{fig:pspline_vs_tspline}
\end{figure}

\medskip{}
\noindent{\bf Our contributions.} To overcome the aforementioned issues, we propose in Section~\ref{scn:transport_spline_alg} a new method for constructing measure-valued splines. Our method outputs Monge solutions, and moreover enjoys significant computational advantages: it only requires $N$ evaluations of Monge maps and standard Euclidean cubic spline fitting to output trajectories. In the case where all of the measures are Gaussian, our approach is more interpretable and scalable than the SDP-based approach of~\textcite{chenMeasurevaluedSplineCurves2018}.

In particular, for Gaussian measures, our method only requires one $d \times d$ matrix inversion and $O(1)$ multiplications per sample point $\mus_{t_i}$.
In comparison, the method of~\textcite{chenMeasurevaluedSplineCurves2018} solves an SDP with $N$ coupled $4d \times 4d$ matrix variables. In the general case we still only need to perform $N$ pairwise OT computations, which can be done efficiently~\parencite{AltWeeRig17}, while the competing algorithms in~\textcite{benamouSecondOrderModels2018} require time exponential in either $N$ or $d$.

 Our new method comes with a theoretical study of its approximation error. In the Gaussian setting, we introduce new techniques for studying quantitative approximation of transport maps and vector fields. In turn, it yields an approximation guarantee analogous to the classical setting (Theorem~\ref{thm:approx_thm}), but adapted to the geometry of the space. This paves the way for a principled theory of approximation on Wasserstein space that mirrors classical Euclidean results. In a forthcoming work, we build upon these ideas to develop higher-order approximation schemes.

A key feature of our approach is its flexibility, which allows us to easily extend our method to fitting thin-plate splines for measures indexed by high-dimensional covariates. We study the case of two-dimensional spatial covariates in Section~\ref{scn:thin_plate}.

\medskip{}
\noindent{}\textbf{Notation.} For a curve such as ${(\mu_t)}_{t\in [0,1]}$ or ${(X_t)}_{t\in [0,1]}$, defined over $[0,1]$, we use the concise notations $(\mu_t)$ and $(X_t)$ respectively, where the time variable $t$ is always understood to range over the interval $[0,1]$.

\section{\allcaps{Background on Optimal Transport}}\label{scn:ot}
In this section, we recall useful notions from optimal transport and provide some of the key theory used for Wasserstein splines. We refer readers to the standard textbooks~\textcite{villani2003topics, villani2009ot, santambrogio2015ot} for introductory treatments.

Given two probability measures $\mu_0$, $\mu_1$ on $\R^d$ with finite second moment, the $2$-Wasserstein distance  $W_2$ is defined as
\begin{align}\label{eq:w2}
W_2^2(\mu_0, \mu_1) := \inf_{\pi \in \Pi(\mu_0, \mu_1)} \int \|x - y\|^2 \, \D\pi(x, y),
\end{align}
where $\Pi(\mu_0, \mu_1)$ is the set of all joint distributions with marginals
$\mu_0$ and $\mu_1$. This indeed defines a distance 
on probability measures with finite second moment,
and we denote the resulting metric space by $\mc P_2(\R^d)$.
If $\mu_0$ has a density with respect to Lebesgue measure, then the solution of~\eqref{eq:w2} is unique, and it is supported on the graph of a function $T \colon \R^d \to \R^d$, called the {\it Monge map}. Moreover, it is characterized as the unique mapping such that (i) the pushforward of $\mu_0$ via $T$ is $\mu_1$ and (ii) there exists a convex function $\phi \colon \R^d \to \R \cup \infty$ such that $T = \nabla \phi$. That is, if $X_0 \sim \mu_0$, the solution of~\eqref{eq:w2} is the law of $\left(X_0,\nabla \phi(X_0)\right)$. For the rest of the paper, without further comment, we work exclusively with  probability measures that admit a density and have a finite second moment.



It has been understood since the seminal work of Otto that $\mc P_2(\R^d)$ exhibits many of the properties of a Riemannian manifold, a fact which has been instrumental to applications of optimal transport to partial differential equations~\parencite{jko1998, carrillovaes2019covfokkerplanck}, sampling~\parencite{bernton2018jko, durmusetal2019langevin, lulunolen2019birthdeath, chewietal2020mirror, chewietal2020svgd}, and barycenters~\parencite{backhoffetal2018barycenters, zemelpanaretos2019barycenters, chewietal2020buresgd}.
Specifically, given a regular curve ${(\mu_t)}$, there is a well-defined notion of a ``tangent vector'' $v_t$ to the curve at time $t$.
This is a vector {\it field} of instantaneous particle velocities, where $\mu_t$ is interpreted as the law of the particles at time $t$. The field $v_t$ arises from optimally coupling the curve at nearby times, and we have the limiting result
\begin{equation} \label{eq:vmaplimit}
    v_t = \lim_{h \to 0} \frac{T_{\mu_t \to \mu_{t+h}} - \id}{h}\qquad \text{in}~L^2(\mu_t)
\end{equation}
where $T_{\mu_t \to \mu_{t+h}}$ is the Monge map between $\mu_t$ and $\mu_{t+h}$. For a proof see~\textcite[Proposition 8.4.6]{ambrosio2008gradient}.

This differential structure has been especially useful in fluid dynamics, by connecting the equivalent Eulerian and Lagrangian perspectives on particle flows.
The former keeps track of the density $\mu_t$ and velocity $v_t$ of particles passing through any given time and spatial position. In contrast, the Lagrangian perspective tracks the trajectories of individual particles, which can be obtained as integral curves of the velocity fields; that is, we solve the ODE
\begin{equation*}
    \dot{X}_t = v_t(X_t), \qquad X_0 \sim \mu_0.
\end{equation*}
Chosing the vector fields $v_t$ to be the tangent vectors above precisely yields that $X_t\sim \mu_t$. Thus, the Lagrangian perspective associates a natural stochastic process, ${(X_t)}$,  with the curve of measures ${(\mu_t)}$; we therefore refer to the process $(X_t)$ as the \emph{Lagrangian coupling}.
See~\textcite[\S 5.4]{villani2003topics} for further details.

\section{\allcaps{Splines on Euclidean Space, Manifolds, and Wasserstein Space}}\label{scn:splines}

We recall the definition of \emph{natural cubic splines}. 
Given points $(x_0, x_1, \ldots, x_N) \subset \R^d$ to interpolate at a sequence of times
$0 = t_0 < t_1 < \cdots < t_N = 1$, consider the variational problem
\begin{equation}\label{eq:cubic_spline}
    \min_{(\gamma_t)} \int_0^1 \norm{\ddot{\gamma}_t}^2 \, \D t  \;\;\; \text{s.t.}\; \;\; \gamma_{t_i} = x_i ~\text{for all}~i.
\end{equation}
The solution to this minimization problem is a piece-wise cubic polynomial that
is globally $\mc C^2$ and has zero acceleration
at times $t_0 = 0$ and $t_N = 1$.

Based on this energy-minimizing property, there is a natural generalization of cubic splines to Riemannian manifolds: in~\eqref{eq:cubic_spline} the acceleration $\ddot \gamma$ is replaced with its Riemannian analogue, the covariant derivative $\nabla_{\dot \gamma} \dot \gamma$ of the velocity, and the norm $\norm\cdot$ is given by the Riemannian metric.
However, unlike its Euclidean counterpart, there is no general algorithm to fit Riemannian cubic splines, leading to alternative proposals~\parencite{gousenbourgermassartabsil2019bezier}.

In addition to a first-order differentiable structure (the tangent space),~\textcite{gigli2012secondorder} has developed a second-order calculus on $\mathcal{P}_2(\R^d)$, including a covariant derivative $\nabla$. Thus, in analogy with the Riemannian setting, we can define \emph{energy splines} (\emph{E-splines} in short) via
\begin{equation}\label{eq:e_spline}
    \inf_{(\mu_t, v_t)} \int_0^1 \norm{\nabla_{v_t} v_t}_{L^2(\mu_t)}^2 \, \D t ~\text{ s.t. } \mu_{t_i} = \mus_{t_i}~\text{for all}~i
\end{equation}
where the minimization is taken over all curves $(\mu_t)$ and their tangent vectors $(v_t)$ (see Section~\ref{scn:ot}). The solution to this problem naturally yields a stochastic process ${(\Y_t)}$ with  marginal laws ${(\mu_t)}$, namely: we draw $\Y_0 \sim \mu_0$, and conditioned on $\Y_0$ the rest of the trajectory is determined by the ODE $\dot \Y_t = v_t(\Y_t)$.

E-splines were introduced concurrently by~\textcite{chenMeasurevaluedSplineCurves2018, benamouSecondOrderModels2018}. Since E-splines are intractable, these authors proposed a relaxation which we call \emph{path splines} (\emph{P-splines} in short):
\begin{align}\label{eq:p_spline}
    \inf_{(\Y_t)} \int_0^1 \E[\norm{\ddot \Y_t}^2] \, \D t,
\end{align}
where the infimum is taken over stochastic processes $(\Y_t)$ with values in $\R^d$ and such that $\Y_{t_i} \sim \mus_{t_i}$ for all $i=0,1,\dotsc,N$. (This is indeed a relaxation in a formal sense detailed in the papers referenced above.)  The name derives from the fact that this is an optimization over measures in path space, and the problem~\eqref{eq:p_spline} can be reduced to a multimarginal optimal transport problem with quadratic cost. 

Unfortunately, though solvable in principle, the formulation~\eqref{eq:p_spline} remains difficult to compute and
its solution is not necessarily induced by a deterministic map; that is, there is no guarantee of a deterministic function $\phi_t : \R^d\to\R^d$ such that $X_t = \phi_t(X_0)$. This point is particularly problematic for inference of trajectories as illustrated in Figure~\ref{fig:pspline_vs_tspline}.

Given the various definitions of splines, some natural questions arise.
Specifically, the papers above left open the question of whether E-splines coincide with P-splines, and whether the solution to the P-spline problem is necessarily induced by Monge maps.
We conclude this section by resolving these questions in the negative.

\begin{prop}[informal] \label{prop:non_Monge_p_spline}
There exist non-degenerate Gaussian data
$\mus_{t_0},\mus_{t_1},\dotsc,\mus_{t_N}$ such that there
is a unique jointly Gaussian solution to the P-spline problem~\eqref{eq:p_spline} and it is \emph{not} induced by a deterministic map.
\end{prop}

\begin{prop}[informal] \label{prop:e_spline_p_spline}
There exist non-degenerate Gaussian
data $\mus_{t_0},\mus_{t_1},\dotsc,\mus_{t_N}$ for which the E-spline~\eqref{eq:e_spline} and P-spline~\eqref{eq:p_spline} interpolations do not coincide.
\end{prop}

Investigation of these questions requires some care, since there are many subtleties regarding the definitions. We give a careful discussion and proofs in Appendix~\ref{appendix:example}.

\section{\allcaps{Transport Splines}}\label{scn:transport_splines}

\subsection{The Algorithm}\label{scn:transport_spline_alg}

To address the difficulties discussed in the previous section, we propose a new method for measure interpolation, which we call \emph{transport splines}.
Our framework decouples the interpolation problem into two steps:
\begin{enumerate}
    \item Couple the given measures, that is, construct a random vector $(\Y_{t_0},\Y_{t_1},\dotsc,\Y_{t_N})$ with the specified marginal laws $\mus_{t_0}, \mus_{t_1},\dotsc,\mus_{t_N}$.
    \item Apply a Euclidean interpolation algorithm to the points $\Y_{t_0},\Y_{t_1},\dotsc,\Y_{t_N}$.
\end{enumerate}
A convenient choice for the second step is to use cubic splines,  but our framework works equally well with other standard Euclidean methods and can be adapted to the application at hand. We illustrate this point in Section~\ref{scn:thin_plate}, where we construct surfaces interpolating one-dimensional measures using thin-plate splines.

A simple and practical choice for the first step, which we explore in the present paper, is to couple the random variables $\Y_{t_0}, \Y_{t_1},\dotsc,\Y_{t_N}$ successively using the Monge maps between them.
That is, we draw $\Y_{t_0} \sim \mus_{t_0}$, and for each $i=1,\dotsc,N$ we set $\Y_{t_i} = T_i(\Y_{t_{i-1}})$, where $T_i$ is the Monge map from $\mus_{t_{i-1}}$ to $\mus_{t_i}$. The second step then reduces to interpolating $\Y_{t_0}, T_1(\Y_{t_0}), \ldots, T_N\circ\cdots\circ T_1(\Y_{t_0})$ in Euclidean space. The interpolation property of transport splines follows readily from the definition of Monge maps since $T_i \circ \dots \circ T_1(\Y_{t_0})\sim \mus_{t_i}$.

For the task of outputting sample trajectories from the transport spline, we summarize our method in Algorithm~\ref{alg:interpolate}, and we display an application to the reconstruction of trajectories in a many-body physical system in Figure~\ref{fig:nbody}. In the next section, we provide detailed motivation for the first step of the algorithm which builds on background from Sections~\ref{scn:ot} and~\ref{scn:splines}.

\begin{algorithm}[H]
  \caption{Sample Transport Spline Trajectories}\label{alg:interpolate}
  \begin{algorithmic}[1]
    \Procedure{interpolate}{${(t_i)}_{i=0}^N$, ${(\mus_{t_i})}_{i=0}^N$}
    \State Draw $\Y_{t_0} \sim \mus_{t_0}$
      \For{$i = 1, \ldots, N$}
      \State Set $\Y_{t_i} = T_i(\Y_{t_{i-1}})$, where $T_i$ is the Monge map from $\mus_{t_{i-1}}$ to $\mus_{t_i}$
      \EndFor
      \State Interpolate the points $\Y_{t_0},\Y_{t_1},\dotsc,\Y_{t_N}$ to obtain a curve ${(\Y_t)}$
      \State \textbf{output} ${(\Y_t)}$
    \EndProcedure   
  \end{algorithmic}
\end{algorithm}

\begin{figure}
\centering
\includegraphics[clip, width=0.78\textwidth]{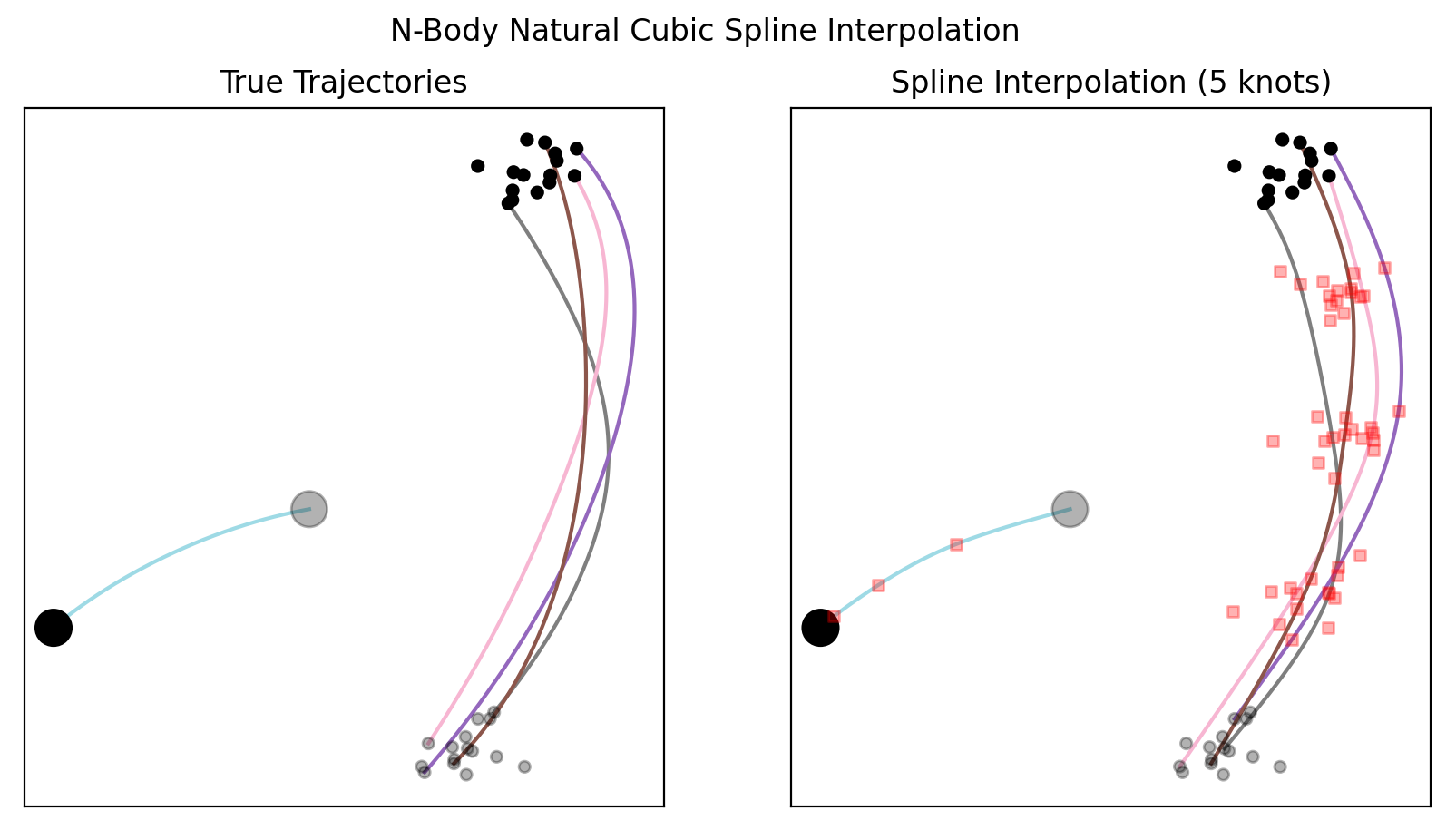}
\vspace{-0em}
\caption{\small Reconstruction of trajectories in a physical system. See Appendix~\ref{appendix:nbody}.}
\label{fig:nbody}
\end{figure}

\subsection{Motivation}

The choice of coupling in the first step of our method is motivated by the geometry of $\mathcal{P}_2(\R^d)$. If the observations $\mus_{t_0}, \ldots, \mus_{t_N}$ sit along a  curve of measures ${(\mus_t)}$, then (as discussed in Section~\ref{scn:ot}) there is an associated Lagrangian coupling ${(\X_t)}$ satisfying $\dot X^\star_t = v_t^\star(\X_t)$. Thus if $\delta = t_1 - t_0$, then $\X_{t_1} = \X_{t_0} + \delta v_{t_0}^\star(\X_{t_0}) + o(\delta)$. On the other hand, from~\eqref{eq:vmaplimit} the Monge map $T_1$ gives a first-order approximation to $v_{t_0}^\star$: $T_1 - \id = \delta v_{t_0}^\star + o(\delta)$ (see ~\textcite[Proposition 8.4.6]{ambrosio2008gradient}). Combining these approximations we get $T_1(\X_0) = \X_{t_1} + o(\delta)$. 
From this heuristic discussion, one expects that as the mesh size $\max_{i\in [N]}(t_i - t_{i-1})$ tends to zero, the coupling $\Y_{t_0},\Y_{t_1},\dotsc,\Y_{t_N}$ obtained via successive Monge maps is a good approximation to the Lagrangian coupling  $(\X_{t_0}, \X_{t_1},\dotsc,\X_{t_N})$.


\subsection{Relationship with E-Splines in One Dimension}

Although E-splines are in general intractable, in the one-dimensional case it turns out that there are many situations of interest in which E-splines coincide with transport splines.
Indeed, suppose that the measures $\mus_{t_0},\mus_{t_1},\dotsc,\mus_{t_N}$ are all one-dimensional, and for a measure $\mu$ let $F_\mu^\dagger$ denote its quantile function.\footnote{Under our assumption that the measures are absolutely continuous, the quantile function $F_\mu^\dagger$ simply coincides with the inverse CDF $F_\mu^{-1}$, but we use the quantile function notation here to reflect the general embedding $\mc P_2(\R) \hookrightarrow L^2[0,1]$.} Let $(G_t)$ be the natural cubic spline in $L^2[0,1]$ interpolating the quantile functions $F_{\mus_{t_0}}^\dagger, F_{\mus_{t_1}}^\dagger,\dotsc,F_{\mus_{t_N}}^\dagger$. Then:

\begin{thm}\label{thm:e_vs_transport_1d}
Suppose that for all $t$, $G_t$ is a valid\footnote{A valid quantile function $G_t \colon [0,1] \to \R \cup \{\pm\infty\}$ is increasing and right-continuous.} quantile function. Then the transport spline and the E-spline~\eqref{eq:e_spline} both coincide with the curve $(\mu_t)$ where $\mu_t$ has quantile function $G_t$. Furthermore, if $(X_t)$ is the stochastic process associated with the transport spline and $(X_t^\star)$ is the Lagrangian coupling for the E-spline, then $(X_t)$ and $(X_t^\star)$ have the same distribution as the law of $(G_t(U))$, where $U$ is a uniform random variable on $[0, 1]$.



\end{thm}


We emphasize that, in light of the counterexamples described at the end of Section~\ref{scn:splines}, the P-spline and E-spline are likely to differ generically and, in fact, they differ in the Gaussian case, which is covered by the above theorem (see Appendix~\ref{appendix:e_and_p_different}). Therefore, it appears that the transport spline is more suitable as a relaxation of the E-spline when interpolating univariate distributions.

We give the proof of Theorem~\ref{thm:e_vs_transport_1d} in Appendix~\ref{appendix:e_vs_t_1d}.

\section{\allcaps{The Gaussian Case}}\label{scn:gaussian_case}
We now focus on the Gaussian case and we assume that we employ natural cubic splines in Step~2 of our algorithm. For simplicity, we can assume that the measures are centered.\footnote{The discussion here extends easily to incorporate non-centered measures.} A centered non-degenerate Gaussian can be identified with its covariance matrix, and the Wasserstein distance induces a Riemannian metric on the space of positive definite matrices. The resulting manifold is called the Bures-Wasserstein space~\parencite[after][]{bures1969}; see~\textcite{bhatiajainlim2019bures} for a comprehensive survey.

\subsection{Gaussian Transport Splines}

It is known that the Monge map from Gaussian $\Normal(0, \Sigma_1)$ to $\Normal(0, \Sigma_2)$ is the linear map $T$ given by
\begin{equation}\label{eq:gaussian_ot_map}
    T(X) = \Sigma_1^{-1/2} \bigl(\Sigma_1^{1/2} \Sigma_2 \Sigma_1^{1/2} \bigr)^{1/2} \Sigma_1^{-1/2} X
\end{equation}


Cubic splines have the property that the interpolation evaluated at time $t$ is a linear function of the interpolated points ${(x_{t_i})}_{i=0}^N$. That is, there is a linear map $S_t$ (indexed by time) such that $t \mapsto S_t(x_{t_0}, \ldots, x_{t_N})$ is the cubic spline interpolating the data.\footnote{Note that the matrix $S_t$ is independent of ${(x_{t_i})}_{i=0}^N$, but depends on the time grid ${(t_i)}_{i=0}^N$.}

This fact follows from the discussion in Appendix~\ref{appendix:cubic_splines} and it has important consequences for our algorithm:
\begin{enumerate}
    \item It implies that our algorithm outputs a process $(\Y_t)$ such that $\Y_t$ is a linear function of $\Y_{t_0},\Y_{t_1},\dotsc,\Y_{t_N}$.
    On the other hand, each $\Y_{t_i}$ is a linear function of $\Y_{t_0}$, which follows from the description of Step 1 of our algorithm and the fact that Monge maps between Gaussians are linear~\eqref{eq:gaussian_ot_map}.
    
    Since a linear function of a Gaussian is also Gaussian, we conclude that \emph{the transport spline interpolating Gaussian measures only passes through Gaussian measures}.
    \item From the previous point, it is clear that the covariance matrix of $\Y_t$ can be computed in terms of $S_t$, $\Sigma_{t_0}$, and the Monge maps (which have the closed-form expression~\eqref{eq:gaussian_ot_map}).
    We conclude that in this setting, not only can we output sample trajectories as in Algorithm~\ref{alg:interpolate}, but \emph{we can also efficiently output the covariance matrices of the interpolated measures}.
\end{enumerate}

Furthermore, this discussion extends to any other interpolation method with this linearity property, such as higher-order splines, polynomial interpolation, and thin-plate splines.

We also remark that in the case where the data consists of \emph{one-dimensional} Gaussian distributions, then in many cases the transport spline and the E-spline (described in Section~\ref{scn:splines}) coincide.

\begin{prop}\label{prop:gaussian_e_vs_t}
    Suppose that $\mus_{t_0}, \mus_{t_1},\dotsc,\mus_{t_N}$ are one-dimensional Gaussians. Then, if the transport spline $(\mu_t)$ interpolating these data is never degenerate, i.e., $\mu_t$ is a non-degenerate Gaussian for each $t \in [0,1]$, then the conditions of Theorem~\ref{thm:e_vs_transport_1d} hold.
\end{prop}

As discussed above, the transport spline through Gaussians automatically remains Gaussian, so the only hypothesis to check in this proposition is the non-degeneracy. See Appendix~\ref{appendix:e_vs_t_1d} for a discussion.

\subsection{Approximation Guarantees}\label{scn:approx_guarantee}

Our method is the first to provide approximation guarantees on Wasserstein space. In order to obtain strong quantitative results, we focus on the Bures-Wasserstein setting detailed in the previous section, where all measures $\mus_{t_i}$ are centered non-degenerate Gaussian distributions.

The Bures-Wasserstein space has already been used in works such as~\textcite{modin2017matrixdecomposition, chewietal2020buresgd} as a prototypical setting in which to understand the behavior of algorithms set on the general Wasserstein space. Although the Bures-Wasserstein space is a Riemannian manifold and transport splines can in principle be studied using purely Riemannian techniques, we give proofs inspired by optimal transport so that the analysis may be more easily extended to other settings of interest.

We now state our main approximation result.

\begin{thm}\label{thm:approx_thm}
Let ${(\mus_t)}$ be a curve of measures in Bures-Wasserstein space, and let ${(\X_t)} \sim {(\mus_t)}$ be the Lagrangian coupling.
Let:
\begin{itemize}
    \item $L := \sup_{t\in [0,1]}{\norm{ \dot{X}^\star_t}_{L^2(\Pr)}}$ be the Lipschitz constant of the curve, and
    \item $R := \sup_{t\in [0,1]}{\norm{\ddot{X}^\star_t}_{L^2(\Pr)}}$ be an upper bound on its curvature, and
    \item $\lambda_{\min}$ be a lower bound on the eigenvalues of the covariance matrices of $\mu_{t_0}^\star,\mu_{t_1}^\star,\dotsc,\mu_{t_N}^\star$.
\end{itemize}

Let ${(\mu_t)}$ be the cubic transport spline interpolating $\mus_{t_0},\dotsc,\mus_{t_N}$ and assume
\begin{align}\label{eq:mesh_condition}
    \alpha \delta
    &\le t_i - t_{i-1} \le \delta, \qquad\text{for}~i=1,\dotsc,N\,,
\end{align}
where $\alpha, \delta>0$.
Then, provided that $\delta < \sqrt{\lambda_{\min}}/(2L)$, we have the following approximation guarantee:
\begin{align*}
    \sup_{t\in [0,1]} W_2(\mu_t,\mus_t) \le \frac{58}{\alpha^3} \, R\delta^2.
\end{align*}
\end{thm}

The proof is given in Appendix~\ref{appendix:approx}.

Some remarks:
\begin{enumerate}
    \item The definition of $L$ in the theorem agrees with the Lipschitz constant of ${(\mus_t)}$ in the metric sense, as can be seen from~\textcite[Theorem 8.3.1]{ambrosio2008gradient}. 
    \item The quantity $\lambda_{\min}^{-1}$ can be interpreted as a bound on the curvature of Bures-Wasserstein space at the interpolation points; see \textcite{massartburescurvature} for details.
    \item The $O(\delta^2)$ rate of convergence is optimal given our assumptions: a bound $R$ on the second covariant derivative of the curve ${(\mus_t)}$. Indeed, this matches classical approximation results for cubic splines on Euclidean space~\parencite{Birkhoffdeboorsplines}. We remark that under these assumptions, piecewise geodesic interpolation, where trajectories are piecewise linear and not differentiable, also achieves the $O(\delta^2)$ rate, and we give the proof of this in Appendix~\ref{appendix:piecewise_geodesic}. Of course, despite achieving the optimal rate in this class of curves, such interpolation is unsuitable for many applications (especially ones in which interpretation and visualization are a priority; see Figure~\ref{fig:sawtooth}).
    \item We did not attempt to optimize the constant factor in Theorem~\ref{thm:approx_thm} and it appears that it can, in fact, be improved;  c.f.\ Remark~\ref{rmk:sharp}  
    \item Cubic splines achieve higher-order approximation rates in the Euclidean setting, albeit over a restricted class of curves. For approximation of functions $f \in \mc C^k, k \le 4$, cubic splines enjoy a $O(\delta^k)$ approximation rate with explicit dependence on $\lVert f^{(k)}\Vert_{\sup}$. It is then natural to ask whether it is possible to obtain rates better than $O(\delta^2)$ through a variant of transport splines. This can indeed be done by using more accurate approximations to the velocity vector fields $(v_t)$; this study will be reported in a forthcoming work.
\end{enumerate}

\section{\allcaps{Thin-Plate Splines}}\label{scn:thin_plate}


To demonstrate the flexibility of our method, we use transport splines to define a class of smooth interpolating surfaces on Wasserstein space. We first recall classical thin-plate splines. For a more complete account see~\textcite{wahbas1990plines}.

Thin-plate splines are the surface analog of cubic splines, and are useful in spatial problems where measurements are taken on a plane. Here, the times $t_i$ are replaced with points $x_i \in \mathbb{R}^2$ at which we observe real values $z_i$. To account for this additional dimension the energy functional $\int_0^1 \norm{\ddot\gamma_t}^2 \, \D t$ that appears in the variational definition~\eqref{eq:cubic_spline} of cubic splines is replaced by its bivariate counterpart. Thin-plate splines are defined as parametrized surfaces $f$ that solve
\begin{equation}\label{eq:tps}
    \inf_f \int_{\R^2} \lVert \nabla^2 f \rVert_{\rm F}^2 ~\text{ s.t.}~ \begin{cases} f : \R^2\to\R \\ f(x_i) = z_i, \; i=0,\dotsc,N \end{cases}
\end{equation}
where $\nabla^2 f$ is the Hessian of $f$, $\lVert \cdot \rVert_{\rm F}$ denotes the Frobenius norm, and the interpolation data $(x_i, z_i) \in \R^2 \times \R$ is given. (Just as before, $f$ is constrained to be $\mathcal{C}^2$.) It can be shown that (\ref{eq:tps}) has a unique solution given by
\begin{equation*}\label{eq:tps_formula}
    f(x) = c_0 + c_1 x^{(1)} + c_2 x^{(2)} + \sum_{i=0}^N \alpha_i \phi( \lVert x - x_i \rVert)
\end{equation*}
where we use $x^{(i)}$ to denote coordinates, and
\begin{equation*}
    \phi(r) = r^2 \log r.
\end{equation*}
This leads to a closed form for the coefficients as follows. Let $K = {(\phi(\lVert x_i - x_j\rVert))}_{i,j=0}^N$ be the ``kernel matrix'' of the data, and define $P \in \R^{(N+1)\times 3}$ to have $i$th row $(1, x^{(1)}_i, x^{(2)}_i)$.\footnote{The function $\phi$ plays the role of a kernel for the reproducing kernel Hilbert space of twice-differentiable, finite-curvature surfaces, but it is {\it not} a kernel because it is not positive definite.} Then let $L \in \R^{(N+4)\times (N+4)}$ be
\begin{equation*}
    L = \begin{bmatrix}
  K  & P \\
  P^\top  &
  0_{3\times 3}
\end{bmatrix}.
\end{equation*}
Letting $b = (z_0, \ldots, z_N, 0, 0, 0)$ be the padded data and $w = (\alpha_0, \ldots, \alpha_N, c_0, c_1, c_2)$ the coefficients from~\eqref{eq:tps_formula}, these solve $    Lw = b$.
This can be inverted explicitly using the Schur complement, and in particular the resulting coefficients are linear in the data ${(z_i)}_{i=0}^N$.

We now consider the measure-valued analog of the interpolation problem, namely, at each point $x_i$ we observe a measure $\mus_{x_i}$ and our goal is to find a smooth interpolating surface $x\mapsto \mu_x$ of measures.

As in the definition of E-splines,~\eqref{eq:tps} can be generalized to Wasserstein space, but it is intractable for the same reasons. In contrast, applying Algorithm~\ref{alg:interpolate} is straightforward. Step 2 simply requires the fitting of a Euclidean thin-plate spline. For Step 1 we need only produce couplings between the observed measures $\mus_{x_i}$.

One possiblity is to mimic the sequential coupling technique described in Section~\ref{scn:transport_spline_alg}, namely we fix the ordering $x_0,x_1,\dotsc,x_N$ and use the system of Monge maps $T_{i-1,i}$ taking $\mus_{x_{i-1}}$ to $\mus_{x_i}$. As before, we can draw $\Y_{x_0} \sim \mus_{x_0}$ and then successively compute the random variables $\Y_{x_i} = T_{i-1,i}(\Y_{x_{i-1}}) \sim \mus_{x_i}$ for all $i$. Sequential coupling is unsuitable here, however, because it distorts the geometry of the plane. To circumvent this issue, we next turn towards the special case when the measures $\mu_{x_i^\star}$ are defined over $\R$, which is already interesting enough to capture a breadth of applications.

The study of $\mc P_2(\R)$ is greatly simplified by the fact that it is isometric to a convex subset of a Hilbert space and is therefore \emph{flat}.
Indeed, the special structure of $\mathcal{P}_2(\R)$ has already been used fruitfully in many prior applications of optimal transport, such as curve registration~\parencite{panaretos2016}, geodesic principal components~\parencite{geodesicpcaw2}, estimation of barycenters~\parencite{bigotbarycenters1d}, and uncoupled isotonic regression~\parencite{rigolletweedisotonic}.

For our purposes, we will use the following key property of $\mc P_2(\R)$: there is a unique coupling of all of the measures $\mus_{x_0}, \mus_{x_1},\dotsc,\mus_{x_N}$ which is \emph{simultaneously} optimal for every pair of measures. In other words, there exist random variables $\Y_{x_0},\Y_{x_1},\dotsc,\Y_{x_N}$ such that for any $i, j =0,1,\dotsc,N$, we have $\Y_{x_j} = T_{i,j}(\Y_{x_i})$, where $T_{i,j}$ is the Monge map from $\mus_{x_i}$ to $\mus_{x_j}$. Sampling from this coupling can be done using either of the of the following equivalent procedures:
\begin{enumerate}
    \item Draw $\Y_{x_0} \sim \mus_{x_0}$, and for each $i\in [N]$ let $\Y_{x_i} = T_{0,i}(\Y_{x_0})$ (the choice of $x_0$ does not affect the coupling).
    \item Draw a uniform random variable $U$ on $[0,1]$, and for $i=0,1,\dotsc,N$ set $\Y_{x_i} = F_{\mus_{x_i}}^{-1}(U)$, where $F_\mu$ denotes the CDF of $\mu$.
\end{enumerate}
See Appendix~\ref{appendix:1d_coupling} or~\textcite[\S 2.1-2.2]{santambrogio2015ot}.

In Figure~\ref{fig:caltemp} we display an application of thin-plate transport splines to temperature data. In the left-hand column we plot the quantiles of the interpolated measures. This is especially convenient when all of the measures are Gaussian, in which case there is a simple and efficient algorithm for computing these quantiles (see Appendix~\ref{appendix:quantile}). The details for the experiment are given in Appendix~\ref{appendix:caltemp}.

\begin{figure}
\centering
\hspace*{-13.5cm}
\includegraphics[clip, width=1.75\textwidth]{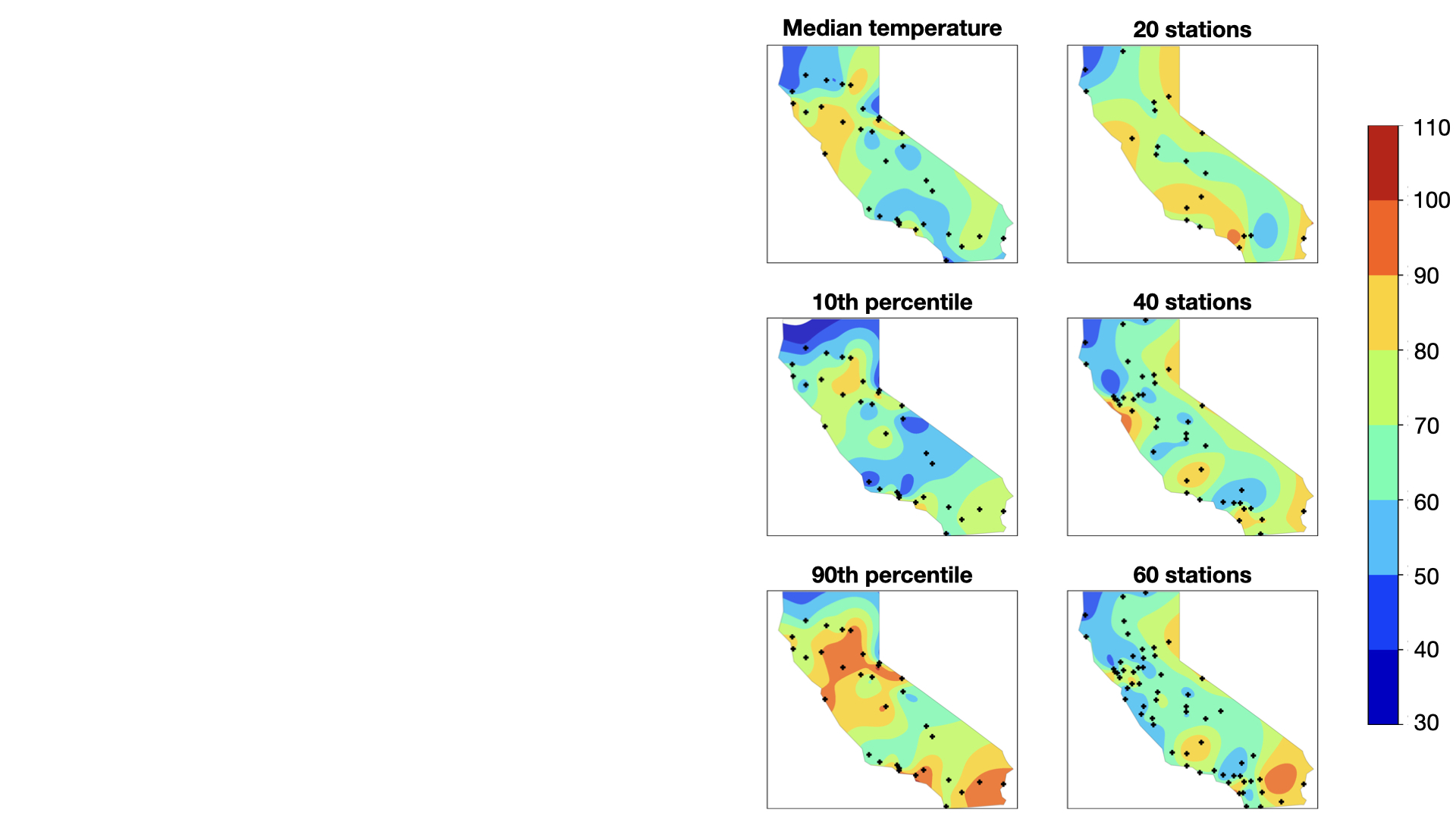}
\vspace{-0.25em}
\caption{\small Thin-plate splines for California temperature data (in ${}^\circ \text F$); in the left column are the quantiles, while in the right are the means of the interpolated measures for an increasing sample of observations. See Appendix~\ref{appendix:caltemp}.}
\label{fig:caltemp}
\end{figure}

We conclude this section with a few remarks about the case of higher-dimensional measures, in which case there is no simultaneous optimal coupling of the measures. If we wish to use Monge map couplings as in Algorithm~\ref{alg:interpolate}, one possibility is to first construct a tree graph whose vertices are the data $\mus_{x_i}$, and use Monge map couplings along the edges of the tree. Here, the tree should be chosen to adequately capture the two-dimensional geometry of the spatial covariates. This consideration becomes especially relevant when the spatial covariates are sampled from a manifold, and it is of interest to combine our methodology with existing results on approximation of manifolds via graphs~\parencite{singer2006graphlaplacian}.


\section{\allcaps{Open Questions}}

We conclude by discussing some interesting directions left open in this work. A natural question is to develop a computationally tractable notion of \emph{smoothing} splines, and to investigate its statistical properties in the context of Wasserstein regression where the $\mus_{t_i}$ are observed with noise. As a second question, we remark that an approximation guarantee such as Theorem~\ref{thm:approx_thm} can be compared with quantitative stability results for Monge maps~\parencite{gigli2011holder, HutRig19} and extending such results to general Wasserstein space will likely require new techniques.

\bigskip{}
\textbf{Acknowledgments}.
Sinho Chewi and Austin J.\ Stromme were supported by the Department of
Defense (DoD) through the National Defense Science \& Engineering Graduate Fellowship (NDSEG)
Program.
Julien Clancy and George Stepaniants were supported by the NSF GRFP. This material is based upon work supported by the National Science Foundation Graduate Research Fellowship under Grant No.\ 1745302.
Thibaut Le Gouic was supported by ONR grant N00014-17-1-2147 and NSF IIS-1838071.
Philippe Rigollet was supported by NSF awards IIS-1838071, DMS-1712596, DMS-T1740751, and DMS-2022448.







\appendix

\section{\allcaps{Details for the P-Spline Example}}\label{appendix:example}

In this section we consider the following spline problem:
for $N > 1$ and
times $t_i = i/N$, $i = 0, \ldots, N$,
suppose we observe
\begin{equation}\label{eqn:indep_interp_gaussian}
\mus_{t_i} := \Normal\bigl(0, (1 - t_i)^2 + t_i^2\bigr), \quad i = 0, \ldots, N.
\end{equation}
This is the data for which we make the claims in Propositions~\ref{prop:non_Monge_p_spline}
and~\ref{prop:e_spline_p_spline}.


\subsection{Proposition~\ref{prop:non_Monge_p_spline}}\label{appendix:non_Monge_p_spline}

We begin by remarking that in general, there is no reason to expect that solutions of the P-spline problem~\eqref{eq:p_spline} are deterministic.
Indeed, consider the following.

\begin{prop}
Let $\mus_0$ and $\mus_1$ be any probability measures. Then, \emph{any} coupling $(\Y_0, \Y_1)$ of the two measures induces an optimal P-spline solution $(\Y_t)$ to~\eqref{eq:p_spline} with data $\mus_0$ and $\mus_1$.
\end{prop}
\begin{proof}
    Indeed, simply set $\Y_t := (1-t) \Y_0 + t\Y_1$. Since $t\mapsto \Y_t$ is a line traversed at constant speed, it incurs zero P-spline cost and is therefore optimal for~\eqref{eq:p_spline}. \placeqed
\end{proof}

As this example shows, the P-spline problem with two measures is quite degenerate; in particular, it does not recover the $W_2$ geodesic joining $\mu_0$ to $\mu_1$, and $\X_1$ is not guaranteed to be a deterministic function of $\X_0$. A slight modification of this simple example yields:

\begin{prop}
Let $\mus_0$ be any absolutely continuous measure. Then, there exist absolutely continuous data ${(\mus_{i/N})}_{i=1}^N$ and an optimal solution $(\Y_t)$ to the P-spline problem~\eqref{eq:p_spline} for ${(\mus_{i/N})}_{i=0}^N$ such that $\Y_1$ is not a deterministic function of $\Y_0$.
\end{prop}
\begin{proof}
    Indeed, let $T, \bar T : \R^d\to\R^d$ be two mappings which are $\mus_0$-a.e.\ distinct, i.e., $T \ne \bar T$.
    Draw $\Y_0 \sim \mus_0$.
    Then, we either set $\Y_t = (1-t) \Y_0 + tT(\Y_0)$ or else $\Y_t = (1-t)\Y_0 + t\bar T(\Y_0)$ with probability $1/2$ each (with the choice being made independently of the draw of $\Y_0$).
    Set $\mus_{i/N} := \on{law}(\Y_{i/N})$.
    
    By construction, the marginals of the process $(\Y_t)$ at times $0,1/N,\dotsc,1$ do indeed interpolate the data.
    Also, since $t\mapsto \Y_t$ is a straight line traversed at constant speed, then $(\Y_t)$ incurs zero P-spline cost and is optimal for~\eqref{eq:p_spline}.
    
    Since $T$ and $\bar T$ are distinct, $\Y_1$ is not a deterministic function of $\Y_0$. Also, the mappings $T$ and $\bar T$ can easily be chosen to make the data all absolutely continuous (e.g., by taking them to be gradients of uniformly convex functions; c.f.\ the proof of~\textcite[Proposition 5.9]{villani2003topics}). \placeqed
\end{proof}

(Compare this with Proposition 7 and the subsequent remark in~\textcite{benamouSecondOrderModels2018}.)

We next turn towards the Gaussian case. As detailed in~\textcite{chenMeasurevaluedSplineCurves2018, benamouSecondOrderModels2018}, the P-spline problem~\eqref{eq:p_spline} can be reduced to a multimarginal optimal transport problem involving the measures $\mus_{t_0},\mus_{t_1},\dotsc,\mus_{t_N}$,
\begin{align}\label{eq:multimarginal}
    \inf_{\pi \in \Pi(\mus_{t_0},\mus_{t_1},\dotsc,\mus_{t_N})} \int c \, \D \pi,
\end{align}
where $c$ is a quadratic cost function. The reduction is in the following sense: if $\pi$ is an optimal solution for~\eqref{eq:multimarginal}, then let $(\Y_{t_0}, \Y_{t_1},\dotsc,\Y_{t_N}) \sim \pi$, and fit a Euclidean cubic spline $(\Y_t)$ through the points ${(\Y_{t_i})}_{i=0}^N$. Then, the stochastic process $(\Y_t)$ is an optimal solution for~\eqref{eq:p_spline}. Any optimal solution of~\eqref{eq:p_spline} is also of this form, having sample paths that are cubic splines.

Since the cost in the multimarginal problem~\eqref{eq:multimarginal} is quadratic, it depends only on the mean and covariance matrix of the coupling $\pi$. Suppose now that the data ${(\mus_{t_i})}_{i=0}^N$ is Gaussian, and suppose we are given any optimal coupling $\pi$ for~\eqref{eq:multimarginal}. Then, we can find a \emph{jointly Gaussian} coupling $\bar\pi$ of the data which has the same mean and covariance structure as $\pi$, which means $\bar\pi$ is also optimal for~\eqref{eq:p_spline}. The coupling $\bar\pi$ then induces a Gaussian process $(\bar\Y_t)$ which is optimal for~\eqref{eq:p_spline}. Such a solution has the appealing property that the law $\mu_t$ of $\bar\Y_t$ is also Gaussian for every time $t$.

From this discussion, it is natural to restrict ourselves to solutions to~\eqref{eq:p_spline} which are Gaussian processes. We call such a solution a \emph{Gaussian solution} to the P-spline problem~\eqref{eq:p_spline}.
We now state a counterexample which proves Proposition~\ref{prop:non_Monge_p_spline}.

\begin{prop} \label{prop:non_Monge_p_spline_formal}
Assume $N > 1$. For $i=0,1, \ldots, N$, let $\mus_{t_i}=\Normal(0, (1-t_i)^2 + t_i^2)$. Then there is a unique Gaussian solution to the P-spline problem~\eqref{eq:p_spline} and it is \emph{not} induced by a deterministic map.
\end{prop}

\begin{proof}
The key observation is that the marginals $\mus_{t_i}$ arise
from the curve of measures formed as the law of
$\X_t := (1 - t)\X_0 + t\X_1$ for independent standard Gaussians $\X_0$ and $\X_1$.
If we consider the distribution on paths which is the law
of $(\X_t)$, then it is supported on straight lines traversed at constant speed and so it must be optimal for the $P$-spline problem~\eqref{eq:p_spline}, having zero objective value.

Consider some other stochastic process $(X_t)$
such that the law of
${(X_{t_i})}_{i = 0}^N$ is jointly Gaussian.
For $(X_t)$ to be an optimal solution to the P-spline problem~\eqref{eq:p_spline}, it must also have zero objective value and hence be supported on straight
lines almost surely. Thus, we must have $X_t = (1-t) X_0 + tX_1$. By the marginal constraints we have $\E[X_0^2]=\E[X_1^2]=1$ and so long as $N>1$, for $i = 1, \ldots, N-1$, it holds that $t_i \notin \{0,1\}$ and
\begin{align*}
(1 - t_i)^2 + t_i^2 &= \E\bigl[\big((1-t_i)X_0+t_iX_1\big)^2\bigr]\\
&= (1 - t_i)^2 + t_i^2  + 2t_i\, (1 - t_i) \E[X_0X_1].
\end{align*} Therefore $\E[X_0X_1] = 0$
and 
$(\Y_t)$ has
the same distribution as $(\X_t)$. Consequently,
the unique jointly Gaussian solution to the P-spline
problem is $(\X_t)$. Clearly, the path $(\X_t)$ is not a deterministic function of $\X_0$. Indeed, $\X_1$ is \emph{independent} of $\X_0$. \placeqed
\end{proof}

\begin{rmk}
The uniqueness assertion is false when $N=1$, even when restricting to Gaussian solutions, which again highlights that the P-spline problem between two measures is degenerate.
\end{rmk}

\subsection{Proposition~\ref{prop:e_spline_p_spline}}\label{appendix:e_and_p_different}

In this section we provide the proof of Proposition~\ref{prop:e_spline_p_spline}. Understanding E-splines requires a few technical results, which we first collect before moving on to the proof. We remark that, prior to this work, little was known about E-splines. In particular, it was not known whether the E-spline interpolation of Gaussian measures consists only of Gaussian measures. 

Throughout, it will be convenient to consider the E-spline problem over the  closed convex set of curves taking values in a closed convex set $K$ of a Hilbert space:
    \begin{equation}\label{eq:Cespline}\tag{$\text{E}_K$}
        \min_{\gamma : [0,1] \to K} \int_0^1 \lVert \ddot{\gamma}(t) \rVert^2 \, \D t \quad\text{s.t.}\quad \gamma(t_i) = x_i ~\text{ for all } i
        \end{equation}
Denote by $E[\gamma]=\int_0^1 \lVert \ddot{\gamma}(t) \rVert^2 \, \D t$ the objective function in~\eqref{eq:Cespline}. It follows from the triangle inequality and strict convexity of the function $x \mapsto x^2$  that $E$ is strictly convex on the convex set of admissible curves, so the solution must be unique if it exists. We denote this unique solution by $\gamma_K$.

\begin{prop}\label{prop:e_spline_proj}
    Let $H$ be a Hilbert space, and let $L \subseteq H$ be a closed linear subspace. Take points $x_0, \ldots, x_N \in L$. Then the solution $\gamma_H$ of the E-spline problem (\hyperlink{eq:Cespline}{$\text{E}_H$}) on $H$ satisfies $\gamma_H(t)=\gamma_L(t) \in L$ for all~$t$.
\end{prop}
\begin{proof}

    Let $P$ be the orthogonal projection onto $L$, and suppose $\gamma$ interpolates the points ${(x_i)}_{i=0}^N$. Then for any admissible curve $\gamma(t) = P \gamma(t) + (I - P) \gamma(t)$, so $\ddot{\gamma}(t) = P \ddot{\gamma}(t) + (I - P)\ddot{\gamma}(t)$ as well. Since these two terms are orthogonal, we have
    \begin{equation*}
        \lVert \ddot{\gamma}(t) \rVert^2 = \lVert P \ddot{\gamma}(t) \rVert^2 + \lVert (I - P) \ddot{\gamma}(t) \rVert^2.
    \end{equation*}
    Thus, on the one hand, if $\bar \gamma (t) = P \gamma_H(t)$ then $E[\bar \gamma] \leq E[\gamma_H]$, and $\bar \gamma$ is interpolating because $x_i \in L$. On the other hand, $E[\gamma_H]\le E[\gamma_L]\le E[\bar \gamma]$ and by uniqueness, $\gamma_H=\gamma_L$. \placeqed
\end{proof}

\begin{prop}\label{prop:K_to_H}
    Let $K$ be a convex subset of a Hilbert space $H$ whose span is closed, and let $x_1, \ldots, x_n \in K$.  If $\gamma_K(t)$     
    lies in the relative interior of $K$ for all times $t$, then $\gamma_K=\gamma_H$.
\end{prop}
\begin{proof}
    Let $L$ be the linear span of $K$, which is closed. In light of Proposition~\ref{prop:e_spline_proj}, it suffices to prove that $\gamma_K=\gamma_L$ so replacing $H$ by $L$ we may assume that $K$ is of full dimension.
    
    Let $f \colon [0,1] \to H$ be a twice differentiable perturbation such that $f(t_i) = 0$ for all $i$. Hence, $\gamma_K + \varepsilon f$ is admissible for (\hyperlink{eq:Cespline}{$\text{E}_H$}). Since $\gamma_K$ lies in the interior of $K$ and $K$ is full-dimensional, a standard compactness argument shows that for any such $f$ there exists an $\varepsilon > 0$ with $\gamma_K(t) + \varepsilon f(t) \in K$ for all $t$. By optimality of $\gamma_K$ we then have $E[\gamma_K + \varepsilon f] \ge E[\gamma_K]$. Thus $\gamma_K$ is stationary for $E$ considered on $H$, and because $E$ is strictly convex it follows that $\gamma_K$ is optimal for (\hyperlink{eq:Cespline}{$\text{E}_H$}) and is therefore equal to $\gamma_H$ by uniqueness. \placeqed
\end{proof}

\begin{prop}\label{prop:GW_is_BW}
    Let $\mus_{t_0}, \mus_{t_1},\dotsc,\mus_{t_N}$ be Gaussian measures on $\R$.
    Consider the Gaussian version of the E-spline problem on $\R$:
    \begin{equation*}\label{eq:bwespline}
        \min_{(\gamma_t)} \int_0^1 \left \lVert \nabla_{v_t} v_t \right \rVert^2_{L^2(\gamma_t)} \, \D t \quad\text{s.t.}\quad\gamma_{t_i} = \mus_{t_i}, \, i = 1, \ldots, N
    \end{equation*}
    where the minimization is taken over curves $(\gamma_t)$ of Gaussian measures with their corresponding tangent vectors $(v_t)$ (as described in Section~\ref{scn:ot}).
    That is, it is the Wasserstein E-spline problem~\eqref{eq:e_spline} in $\mathcal{P}_2(\R)$ with the added constraint that the measures are Gaussian. If there is an optimal solution $(\gamma_t^\star)$ which is a non-degenerate Gaussian for all time, then it is also the solution to the E-spline problem~\eqref{eq:e_spline}.
\end{prop}

\begin{proof}
    It is known that $\mathcal{P}_2(\mathbb{R})$ is isometric to a closed convex subset $S$ of the Hilbert space $H = L^2[0, 1]$ \parencite[see the discussion following][Lemma 9.1.4]{ambrosio2008gradient}. This isometry is given by $\mu \mapsto F_\mu^\dagger$, where $ F_\mu^\dagger$ denotes the quantile function of $\mu$. Let $K$ be  the image of the mean-zero Gaussian measures under this isometry; it is immediate that $K$ is convex, since the Gaussian measures form a geodesically convex set in $\mathcal{P}_2(\R)$, and it has closed span because it is finite-dimensional.  In light of this isometry the E-spline problem~\eqref{eq:e_spline} is equivalent to  (\hyperlink{eq:Cespline}{$\text{E}_S$}) while the Gaussian E-spline problem stated in the proposition is equivalent to (\hyperlink{eq:Cespline}{$\text{E}_K$}) and $\gamma^\star=\gamma_K$ (the preservation of E-splines under isometry is discussed in Appendix~\ref{appendix:splines_isometry}).
    
  Applying Proposition~\ref{prop:K_to_H} to $\gamma^\star=\gamma_K$, we deduce that $\gamma^\star = \gamma_H$. Moreover, $E[\gamma_H] \le E[\gamma_S] \le E[\gamma^\star]$, whence by uniqueness we get that $\gamma^\star = \gamma_S$ as well. \placeqed
\end{proof}

We also require a technical lemma regarding P-splines which remain Gaussian for all times, which follows from considerations of several-variable complex functions.
\begin{lem}\label{lem:pspline_is_jointgaussian}
    Let $(\mu_t)$ be a P-spline with initial and final data $\mu_0$ and $\mu_1$ which are Gaussian, and assume:
    \begin{enumerate}
        \item $\mu_t$ is a Gaussian distribution for all times $t$,
        \item $(\mu_t)$ has zero cost for the P-spline objective.
    \end{enumerate}
    Then $(\mu_t)$ is induced by a jointly Gaussian coupling of $\mu_0$ and $\mu_1$.
\end{lem}
\begin{proof}
    Since $(\mu_t)$ has zero cost it must be supported on straight lines, so if we let $X_t \sim \mu_t$ where these are coupled according to the $(\mu_t)$ coupling, then
    \begin{equation}\label{eq:lem_intermediate_gaussian}
        X_t = (1-t)X_0 + t X_1
    \end{equation}

    and by assumption this variable is Gaussian. Let $Z$ be the Gaussian with the same covariance structure as $X$. 
    Scaling \eqref{eq:lem_intermediate_gaussian} by a positive constant, we get, for all $a, b \geq 0$
    \begin{equation*}
        \left\langle (a,b), X \right\rangle \stackrel{\rm d}{=} \left\langle (a,b), Z \right\rangle
    \end{equation*}
    where we mean equality in distribution. This implies
    \begin{equation*}
        \varphi_X(a,b) = \varphi_Z(a,b)
    \end{equation*}
    where $\varphi_Y$ denotes the characteristic function of $Y$  and is defined by $\varphi_Y(z)=\E[ e^{i \langle z, Y \rangle} ]$.
    Now, it is well-known that if $\mathbb{E} e^{m \norm Y} < \infty$ for some $m > 0$ then $\varphi_Y$ continues to a holomorphic function in the strip $\{ z \mid |\text{Im} \, z_i| < m \hspace{0.5em} \forall i \}$ \parencite[Theorem 2.7.1]{lehmannromamo_hypotheses}. In particular, if $Y$ has sub-Gaussian tails, $\varphi_Y$ is entire.
    
    Functions of several complex variables admit an identity theorem, similar to the univariate complex case, which can be found in~\textcite[Remark 1.20]{range_scv}.\footnote{The careful reader will note that the hypothesis of this theorem is much stronger than the single-variable requirement that $f$ and $g$ agree merely on a set with an accumulation point. For several complex variables this is not sufficient; indeed, several-variable holomorphic functions never have isolated zeros.} This is:
    \begin{thm*}[identity theorem]
        Let $f$ and $g$ be holomorphic functions of several complex variables in a domain $\Omega \subseteq \C^d$, and let $z \in \Omega$. A {\it real cube} of radius $r$ about $z$ is defined as
        \begin{equation*}
            \{(z_1 + x_1, \ldots, z_d + x_d) \in \C^d \mid |\Re x_i| < r~\text{for}~i=1,\dotsc,d\}.
        \end{equation*}
        If $f$ and $g$ agree on a real cube of positive radius about $z$, then $f \equiv g$ on all of $\Omega$.
    \end{thm*}
            Now, $X$ has sub-Gaussian tails. Indeed,
    \begin{align*}
        M_X(t) &= \E e^{\langle t, X \rangle} = \E e^{t_1 X_0 + t_2 X_1} \leq \left( \E e^{2t_1X_0} \, \E e^{2t_2X_1} \right)^{1/2} = e^{t_1^2 \var X_0 + t_2^2 \var X_1}
    \end{align*}
    where $M_X$ denotes the moment generating function of $X$. Thus $\varphi_X$ is entire, along with $\varphi_Z$, and it is clear from the above discussion that they agree on the real cube about $z = (1,1)$ with radius $r = 1$. The identity theorem then implies that $\varphi_X \equiv \varphi_Z$, so $X \stackrel{\rm d}{=} Z$. Thus $X$ is jointly Gaussian. \placeqed
\end{proof}

Proposition~\ref{prop:e_spline_p_spline} is implied by the following result.
\begin{prop} \label{prop:e_spline_p_spline_formal}
For $i=0, \ldots, N$, let $\mus_{t_i}=\Normal(0, \sigma_{t_i}^2)$, where $\sigma_t^2=(1-t)^2 + t^2$. Then for all $N \geq 2$, the E-spline~\eqref{eq:e_spline} and P-spline~\eqref{eq:p_spline} interpolations do not coincide.
\end{prop}

Before starting the proof, we dispense with a possible source of confusion. The solution to the P-spline problem~\eqref{eq:p_spline} is a stochastic process $(\Y_t)$; on the other hand, the E-spline solution yields a natural stochastic process, namely the Lagrangian coupling $(\X_t)$ (see Section~\ref{scn:ot}). In the proposition, we are not asserting that the process $(\Y_t)$ and $(\X_t)$ are different (indeed this is an easier statement to prove since the P-spline solution is often not even deterministic; see Appendix~\ref{appendix:non_Monge_p_spline}). Instead, we are asserting that the {\it interpolated measures} associated with the E- and P-splines are different, which is strictly stronger statement.

\begin{proof}
First, the manifold of mean-zero Gaussian measures on $\R$ equipped with the $W_2$ metric is isometric to the ray $[0, \infty)$ equipped with the standard Euclidean metric. Indeed, we have
\begin{equation*}
    W_2\bigl(\mathcal{N}(0,\sigma_0^2), \mathcal{N}(0,\sigma_1^2) \bigr) = |\sigma_0 - \sigma_1|.
\end{equation*}
Suppose we have data $\mus_{t_i} = \mathcal{N}(0,\sigma_i^2)$ at times $t_i$ and let $t\mapsto \gamma(t)$ be the Euclidean spline interpolation of ${(t_i, \sigma_i)}_{i=0}^N$ on $\R$. It is possible that $\gamma(t) \leq 0$ at some $t$, but if $\gamma(t) > 0$ for all $t$, then by Proposition~\ref{prop:K_to_H} it must also be the spline considered on the ray $[0,\infty)$. Since covariant derivatives are preserved under isometry (see Appendix~\ref{appendix:splines_isometry} for a formal verification in our setting), the function $E[\cdot]$ is also preserved under isometry, and so its minimizers --- E-splines --- are preserved as well. This means that the Gaussian-constrained E-spline is
\begin{equation*}
    \mu^{\rm E}_t = \mathcal{N}\bigl(0, \gamma(t)^2\bigr), \qquad t \in [0,1],
\end{equation*}
and by Proposition~\ref{prop:GW_is_BW} this must coincide with the Wasserstein E-spline~\eqref{eq:e_spline}. This is all under the hypothesis that $\gamma(t) > 0$.

Now substitute our example, with $\sigma_i^2 = (1-t_i)^2 + t_i^2$. We need to check that $\gamma(t)$ remains strictly positive for all times. From~\textcite[Theorem 5]{hall_spline_error_bounds}, we see that for all $t$
\begin{equation*}
    |\gamma(t) - \sqrt{t^2 + (1-t)^2}| \leq \frac{5}{384} \cdot 24 \sqrt{2} \cdot \frac{1}{N^4}.
\end{equation*}
For $N \geq 2$ this is less than $0.03$. The smallest value of $\sqrt{t^2 + (1-t)^2}$ is $\sqrt{1/2} \approx 0.7071$, so the spline is bounded below by $0.704$ for all times.

Let $(\mu_t^{\rm P})$ be an interpolating P-spline. It is possible that this is not unique, but if $\mu_t^{\rm P}$ is not Gaussian for some $t$ then we are done, since $\mu_t^{\rm E}$ is Gaussian by Proposition~\ref{prop:GW_is_BW}. Applying Lemma~\ref{lem:pspline_is_jointgaussian}, we see that $\mu_t^{\rm P}$ must be induced by a jointly Gaussian coupling of $\mu_0^\star$ and $\mu_1^\star$, so by Proposition~\ref{prop:non_Monge_p_spline_formal} it must be that $\mu_t^{\rm P} = \mathcal{N}(0, (1-t)^2 + t^2)$.

The standard deviation of $\mu^{\rm E}_t$ is $\gamma(t)$ and this is locally a cubic polynomial in $t$. The standard deviation of the P-spline $\mu_t^{\rm P}$, however, is given by $\sqrt{(1-t)^2 + t^2}$, which cannot be locally represented by a polynomial, so they must differ. \placeqed
\end{proof}

From the final steps of our proof, we see that (in the Gaussian case) P-splines and E-splines will most likely differ generically, since their interpolated variances are polynomial splines of different orders.

\subsection{Preservations of Splines under Isometry}\label{appendix:splines_isometry}

In this section, we give a formal\footnote{The word \emph{formal} here, meaning that the argument proceeds by manipulating the \emph{form} of the expressions, is not a synonym for ``rigorous''.} verification of the assertion that the E-spline functional is preserved under the isometry between $\mc P_2(\R)$ and its image in $H = L^2[0,1]$. Formally, this assertion can be viewed as a manifestation of a classical fact from Riemannian geometry: the covariant derivative (associated with the Levi-Civita connection) depends only on the Riemannian metric, and is thus preserved under isometries.\footnote{In fact, this is related to Gauss's famous \emph{Theorema Egregium}, see~\textcite[\S 4.3]{docarmo2016diffgeo} and~\textcite[Remark 2.7]{docarmo1992riemannian}.}

In the derivation below, we make all necessary regularity assumptions (e.g., we can assume that the measures are compactly supported) in order to convey the intuition. Suppose $(\mu_t)$ is a curve of measures in $\mc P_2(\R)$ and let $(v_t)$ be the corresponding tangent vectors.
The relationship between $(\mu_t)$ and $(v_t)$ is given by the \emph{continuity equation}~\parencite[Theorem 8.3.1]{ambrosio2008gradient}:
\begin{align}\label{eq:cont_eq}
    \partial_t \mu_t + (\mu_t v_t)' = 0.
\end{align}
Here, we use $\partial_t$ for the time derivative, and we use $'$ to denote spatial derivatives. If $F_\mu$ denotes the CDF of $\mu$, then~\eqref{eq:cont_eq} implies
\begin{align*}
    \partial_t F_{\mu_t}(x)
    &= \partial_t \int_{-\infty}^x \D \mu_t
    = -\int_{-\infty}^x (\mu_t v_t)'
    = -\mu_t(x) v_t(x).
\end{align*}
Next, if we differentiate the relation $F_{\mu_t}^{-1}(F_{\mu_t}(x)) = x$, we obtain
\begin{align*}
    0
    &= (\partial_t F_{\mu_t}^{-1})\bigl(F_{\mu_t}(x)\bigr) + (F_{\mu_t}^{-1})'\bigl(F_{\mu_t}(x)\bigr) \\
    &= (\partial_t F_{\mu_t}^{-1})\bigl(F_{\mu_t}(x)\bigr) + \frac{1}{F_{\mu_t}'(x)} \\
    &= (\partial_t F_{\mu_t}^{-1})\bigl(F_{\mu_t}(x)\bigr) + \frac{1}{\mu_t(x)},
\end{align*}
where we have applied the inverse function theorem. Thus,
\begin{align}\label{eq:tangent_space_isom}
    (\partial_t F_{\mu_t}^{-1})(\alpha)
    &= v_t\bigl(F_{\mu_t}^{-1}(\alpha)\bigr).
\end{align}
Differentiating again,
\begin{align*}
    (\partial_t^2 F_{\mu_t}^{-1})(\alpha)
    &= (\partial_t v_t)\bigl(F_{\mu_t}^{-1}(\alpha)\bigr) + v_t'\bigl(F_{\mu_t}^{-1}(\alpha)\bigr) (\partial_t F_{\mu_t}^{-1})(\alpha) \\
    &= (\partial_t v_t + v_t' v_t)\bigl(F_{\mu_t}^{-1}(\alpha)\bigr).
\end{align*}
However, we recognize $\partial_t v_t + v_t' v_t$ as the covariant derivative $\nabla_{v_t} v_t$ in $\mc P_2(\R)$ (see for example the discussion in~\cite[\S 5.1]{chenMeasurevaluedSplineCurves2018}). In particular, it implies
\begin{align*}
    \int_0^1 \abs{\partial_t^2 F_{\mu_t}^{-1}}^2
    &= \int_0^1 \abs{(\partial_t v_t + v_t' v_t) \circ F_{\mu_t}^{-1}}^2 \\
    &= \int \abs{\partial_t v_t + v_t' v_t}^2 \, \D \mu_t \\
    &= \norm{\nabla_{v_t} v_t}_{L^2(\mu_t)}^2,
\end{align*}
where we use the fact that the pushforward of the uniform distribution on $[0,1]$ under $F_{\mu_t}^{-1}$ is $\mu_t$. This equation shows that the norm (measured in $H$) of the acceleration of the curve $t\mapsto F_{\mu_t}^{-1}$ in $H$ is the same as the norm (measured in $\mc P_2(\R)$) of the acceleration of the curve $t\mapsto \mu_t$ in $\mc P_2(\R)$, and thus the E-spline cost functional is preserved by the embedding $\mc P_2(\R) \hookrightarrow H$.

\begin{rmk}
From the equation~\eqref{eq:tangent_space_isom}, we can also read off the isometry between the tangent space of $H$ and the tangent space of $\mc P_2(\R)$.
\end{rmk}

The reader who is uncomfortable with the formal derivation above can instead use the isometric embedding $\mc P_2(\R) \hookrightarrow L^2[0,1]$ as the definition of the geometry of $\mc P_2(\R)$ (and thus, the definition of E-splines on $\mc P_2(\R)$). Indeed, a rigorous development of second-order calculus on Wasserstein space faces significant technical hurdles~\parencite{gigli2012secondorder}, and such a definition is actually more convenient for the purposes of this paper.

\section{\allcaps{E-Splines and Transport Splines in One Dimension}}\label{appendix:e_vs_t_1d}

In this section, we investigate the relationship between transport splines and E-splines on $\mathcal{P}_2(\R)$, leading to a proof of Theorem~\ref{thm:e_vs_transport_1d}.
We will use the calculation in Appendix~\ref{appendix:splines_isometry}, and moreover we recommend that readers read Appendix~\ref{appendix:example} before this section in order to gain familiarity with E-splines.


Recall also that we assume that the measures $\mus_{t_i}$ are absolutely continuous in order to properly define the covariant derivative.
However, the embedding $\mathcal{P}_2(\R)\hookrightarrow L^2[0,1]$ allows us to rigorously extend the definition of an E-spline on all of $\mathcal{P}_2(\R)$.

\begin{proof}[Proof of Theorem~\ref{thm:e_vs_transport_1d}]
    Let $U$ be a uniform random variable on $[0,1]$, and define the random variables
    \begin{align*}
        X_{t_i} := F_{\mus_{t_i}}^\dagger(U) \sim \mus_{t_i}, \qquad i=0,1,\dotsc,N.
    \end{align*}
    From the discussion in Appendix~\ref{appendix:1d_coupling}, these random variables are simultaneously optimally coupled. In particular, each successive pair of these random variables is coupled via a Monge map. It follows from the definition of a transport spline that the stochastic process $(X_t)$ associated with the transport spline can be realized as the (Euclidean) cubic spline interpolating the points ${(X_{t_i})}_{i=0}^N$.
    
    Since each $X_{t_i}$ is a function of $U$, so is the interpolation $X_t$, so we can write $X_t = \tilde G_t(U)$. It follows that $(\tilde G_t)$ is the cubic spline in $H = L^2[0,1]$ which interpolates the quantiles ${\bigl(F_{\mus_{t_i}}^\dagger\bigr)}_{i=0}^N$, that is, $(\tilde G_t) = (G_t)$. At this point, we have established one of the assertions of Theorem~\ref{thm:e_vs_transport_1d}, namely, the explicit description of the process $(X_t)$ associated with the transport spline.
    
    Next, since $X_t = G_t(U)$, by hypothesis $G_t$ is an increasing function that pushes forward the uniform distribution to the law $\mu_t$ of $X_t$. By the characterization of Monge maps in one dimension (Appendix~\ref{appendix:1d_coupling}), it follows that $G_t = F_{\mu_t}^\dagger$.
    
    Since $(G_t)$ is a cubic spline, then it minimizes curvature, i.e., it solves the problem
    \begin{align*}
        \inf_{(G_t)} \int_0^1 \norm{\ddot G_t}_{L^2[0,1]}^2 \, \D t, \quad\text{s.t.}\quad G_{t_i} = F_{\mus_{t_i}}^\dagger~\text{for all}~i.
    \end{align*}
    From our characterization $G_t = F_{\mu_t}^\dagger$, it is clear that $(\mu_t)$ solves the problem
    \begin{align*}
        \inf_{(\mu_t)} \int_0^1 \norm{\partial_t^2 F_{\mu_t}^\dagger}_{L^2[0,1]}^2 \, \D t, \quad\text{s.t.}\quad \mu_{t_i} = \mus_{t_i}~\text{for all}~i,
    \end{align*}
    since the the first problem is a relaxation of the second (given a solution $(\mu_t)$ of the second problem, we can obtain a solution $(G_t) = (F_{\mu_t}^\dagger)$ for the first problem). Indeed, the second problem can be interpreted as the first problem with the additional constraint that the functions $G_t$ must be quantile functions.
    Next, in light of the isometry described in Appendix~\ref{appendix:splines_isometry}, the latter problem is equivalent to
    \begin{align*}
        \inf_{(\mu_t, v_t)} \int_0^1 \norm{\nabla_{v_t} v_t}^2_{L^2(\mu_t)} \, \D t, \quad\text{s.t.}\quad \mu_{t_i} = \mus_{t_i}~\text{for all}~i,
    \end{align*}
    where the infimum is taken over curves $(\mu_t)$ in $\mc P_2(\R)$ and their corresponding tangent vectors $(v_t)$. This problem is seen to be the E-spline problem~\eqref{eq:e_spline}.
    
    We have thus shown that $(\mu_t)$ is an E-spline. Actually, in light of Proposition~\ref{prop:e_spline_proj} and the fact that $(G_t)$ is the spline in $H$, then the E-spline is unique. Thus, the E-spline and transport spline coincide.
    
    Finally, it remains to show that the Lagrangian coupling $(\X_t)$ associated with the E-spline has the same law as $(X_t)$. For this, we can simply appeal to the embedding $\mc P_2(\R) \hookrightarrow H$ again.
    Indeed, since $\dot X_t = \partial_t F_{\mu_t}^\dagger(U)$, the calculation in Appendix~\ref{appendix:splines_isometry} shows that $\dot X_t = v_t(X_t)$ where $(v_t)$ is the tangent vector to $(\mu_t)$, so in fact $(X_t)$ is the Lagrangian coupling of $(\mu_t)$. \placeqed
\end{proof}



In particular, since the Gaussian measures form a 2 dimensional half-subspace of $L^2[0,1]$ with the usual identification $\mathcal{P}_2(\R)\hookrightarrow L^2[0,1]$,
the E-spline interpolation between Gaussian measures is the transport spline if transport splines is not degenerate at any time (i.e., the transport lies in the relative interior of Gaussian measures within $\mathcal{P}_2(\R)$). This yields Proposition~\ref{prop:gaussian_e_vs_t}.


We conclude this section by giving some examples showing that E-splines and transport splines can differ when the spline $(G_t)$ described in Theorem~\ref{thm:e_vs_transport_1d} does not stay within $\mathcal{P}_2(\R)\subset L^2[0,1]$. First, we give a simple Gaussian counterexample.

\begin{prop}
Let $\delta>0$ be sufficiently small and consider the measures
\begin{align*}
    \mus_0 = \mus_1 = \Normal(0,1), \qquad \mus_{1/3} = \mus_{2/3} = \Normal(0, \delta^2).
\end{align*}
Then, the E-spline~\eqref{eq:e_spline} interpolation $(\mu_t^{\rm E})$ and transport spline interpolation $(\mu_t^{\rm T})$ do not coincide for this data.
\end{prop}
\begin{proof}
    Let $(\Y_t)$ denote the stochastic process corresponding to the transport spline.
    It is easy to see that $(\Y_0, \Y_{1/3}, \Y_{2/3}, \Y_1) = (\Y_0, \delta \Y_0, \delta \Y_0, \Y_0)$ is the optimal coupling at the knots.
    If we let $S_t$ denote the linear mapping which produces the spline (as introduced in Section~\ref{scn:gaussian_case}), it follows that \[ \Y_t = S_t(\Y_0, \delta \Y_0, \delta \Y_0, \Y_0) = S_t(1,\delta,\delta,1) \Y_0, \] so that $\mu_t^{\rm T} = \Normal(0, {S_t(1,\delta,\delta,1)}^2)$.
    
    If we identify the space of Gaussians with the half-ray $[0,\infty)$, then the transport spline corresponds to the curve of standard deviations $t\mapsto \abs{S_t(1,\delta,\delta,1)}$. However, because the spline curve $t\mapsto S_t(1,0,0,1)$ becomes negative between $1/3$ and $2/3$, then so does the curve $t\mapsto S_t(1,\delta,\delta,1)$ for small $\delta$.
    It can be checked that at time $1/3$, the curve $t\mapsto \abs{S_t(1,\delta,\delta,1)}$ is not $\mc C^2$ differentiable and therefore cannot be an E-spline. \placeqed
\end{proof}

This counterexample, however, is somewhat degenerate because the transport spline passes through a degenerate measure, and thus it is not clear if the E-spline exists, and if so whether it remains non-degenerate.
We now give another example where the transport spline does not coincide with the E-spline, but the transport spline remains non-degenerate; hence, we believe that the E-spline problem is well-posed for these data.

For this example, we take $\delta > 0$ and let
\begin{align}\label{eq:thibaut_ex}
    \mus_0 = \mus_1 = \text{uniform on}~[-(1+\delta), -1] \cup [1, 1+\delta], \qquad \mus_{1/4} = \mus_{3/4} = \text{uniform on}~[-\delta,\delta].
\end{align}


\begin{figure}[h!]
\centering
\includegraphics[clip, width=0.75\textwidth]{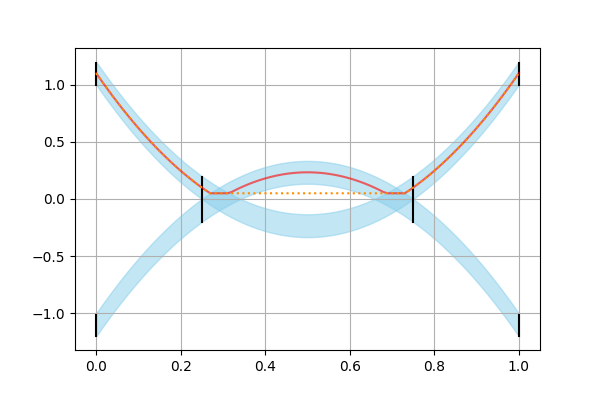}
\caption{\small Transport splines interpolation for the four uniform distributions as in~\eqref{eq:thibaut_ex}. The red line is the quantile of order $3/4$ for the interpolation and the orange dotted line represents the corresponding candidate $\bar F^\dagger_t(u)$ for $u=3/4$ introduced in~\eqref{eq:thibaut_ex_competitor}.}
\label{fig:counterex}
\end{figure}

As in the proof of Proposition~\ref{prop:GW_is_BW}, $\mathcal{P}_2(\R)$ is seen as a convex subset of $L^2[0,1]$ where probability measures are identified as their quantile function.
So our E-spline interpolation can be reformulated as the problem
\[
\inf_{(\mu_t)} \int_0^1 \int_0^1 \norm{\ddot F_t^\dagger(u)}^2 \, \D u \, \D t \qquad\text{s.t.}\qquad \mu_{t} = \mus_{t}~\text{for all}~t\in\{0,1/4,3/4,1\},
\]
where $F_t^\dagger$ denotes the quantile function of $\mu_t$.
In particular, the E-spline interpolation problem can be seen as the transport spline interpolation with the extra constraint that the trajectories of the particles must stay ordered (see Theorem~\ref{thm:e_vs_transport_1d}).

Denote by $(\Y_t)$ the random process given by the transport spline problem.
One can check that 
\[
\Y_t=\text{sign}(\Y_0) \, \Bigl[\frac{16}{3} {(t-1/2)}^2-\frac{1}{3} +\abs{\Y_0}-1\Bigr].
\]

Clearly, for $\delta$ small enough the quantiles $F^\dagger_t(u)$ of order $u>1/2$ associated to the transport spline interpolation decrease before $t=1/4$ and increase after $=3/4$.
In particular, for each $u>1/2$, there exists $1/4<t^-_u<t^+_u<3/4$ such that $\partial_t F^\dagger_t(u)|_{t=t_u^-}=\partial_t F^\dagger_t(u)|_{t=t_u^+}=0$ and $|\partial_t^2 F^\dagger_{t}(u)|>0$ for $t\in(t^-_u,t^+_u)$.
One can check then that the function $u\mapsto \bar F^\dagger_t$ at time $t\in[0,1]$ defined by
\begin{align}\label{eq:thibaut_ex_competitor}
\bar F^\dagger_t(u)=
\begin{dcases}
F^\dagger_{t^-_u}(u), & u \in (t^-_u,t^+_u)\\
F^\dagger_{t}(u), & \text{otherwise}
\end{dcases}
\end{align}
is a quantile function.
In particular, the measures with quantiles $\bar F^\dagger_t$ interpolate the measures~\eqref{eq:thibaut_ex} and 
\[
|\partial_t^2 \bar F^\dagger_t(u)|=
\begin{cases}
0, & u \in (t^-_u,t^+_u)\\
|\partial_t^2 F^\dagger_{t}(u)|, & \text{otherwise},
\end{cases}
\]
ensuring that $\bar F^\dagger_t$ has a lower cost than the transport spline. Thus, the transport spline is not the E-spline. \placeqed

Since the transport spline is non-degenerate for this example, we believe that the E-spline also exists and is non-degenerate. Therefore, we expect that the failure of transport splines to equal E-splines in general is not simply due to the fact that E-splines can be ill-posed.

To summarize: when the trajectories of the transport spline remain ordered throughout the interpolation, then it coincides with the E-spline. Otherwise, there is no reason to expect the two notions of spline to coincide.

\section{\allcaps{Proof of the Approximation Guarantee}}\label{appendix:approx}

Throughout, we assume all random variables are defined on a probability space with probability measure $\Pr$.
Thus, if $X$ is a random variable taking values in $\R^d$, then $\norm X_{L^2(\Pr)} := \sqrt{\E[\norm X^2]}$.

We begin by describing the general strategy for proving the approximation guarantee.
Consider the interval $[t_{i-1}, t_i]$, let $(\X_t)$ denote the Lagrangian coupling for $(\mus_t)_t$, and let $(\Y_t)$ be the stochastic process associated with the transport spline.
Since $\mu_{t_{i-1}} = \mus_{t_{i-1}}$, we can couple the two processes together so that $\Y_{t_{i-1}} = \X_{t_{i-1}}$.
By the definition of the Wasserstein distance, we can bound $W_2(\mu_t,\mus_t) \le \norm{\Y_t - \X_t}_{L^2(\Pr)}$, so it suffices to show that the trajectories $(\Y_t)$ and $(\X_t)$ are close on the interval $[t_{i-1}, t_i]$.

We will use a basic deterministic fact: if two curves $x$ and $y$ defined on $[0,\delta]$ are such that:
\begin{itemize}
    \item $x(0) = y(0)$,
    \item $\dot x(0) = \dot y(0) + O(\delta)$, and
    \item the two curves satisfy the curvature bound \[ \sup_{t\in [0,\delta]}\{\norm{\ddot x(t)} \vee \norm{\ddot y(t)}\} \le R, \]
\end{itemize}
then it follows that $\sup_{t \in [0,\delta]}{\norm{x(t) - y(t)}} \le CR\delta^2$, where $C$ is a numerical constant.

\begin{enumerate}
    \item the velocities of $\Y_t$ and $\X_t$ at time $t=t_{i-1}$ are within $O(\delta)$ of each other (Proposition~\ref{prop:velocity_control});
    \item the trajectory $(\Y_t)$ has curvature $O(R)$ (Proposition~\ref{prop:curv_transport_spline});
    \item the trajectory $(\X_t)$ has curvature $O(R)$;
\end{enumerate}
The last step is immediate from our assumptions; the point of the second step is to control the curvature of the interpolated process ${(\Y_t)}$ in terms of the curvature of the true process ${(\X_t)}$.

Putting these pieces together, we give the proof of Theorem~\ref{thm:approx_thm} in Appendix~\ref{appendix:finally_the_proof}.

\subsection{Notation}

Since we study the approximation guarantee in the Bures-Wasserstein setting, we can equivalently think in terms of the probability measure (a Gaussian), or in terms of the covariance matrix. It will be useful to employ the language of matrices, so we fix notational conventions here.

Associated with the curve ${(\mus_t)}$, we have a corresponding curve of covariance matrices ${(\Sigma_t)}$ such that $\mus_t = \Normal(0, \Sigma_t)$.

Given a matrix $A \in \R^{d\times d}$, we define the norm
\begin{align*}
    \norm A_\Sigma
    &:= \sqrt{\langle A, \Sigma A\rangle}.
\end{align*}
The norm is defined so that if $\X \sim \Normal(0, \Sigma)$, then $\norm{A\X}_{L^2(\Pr)} = \norm A_\Sigma$. From our eigenvalue bound we have $\norm A_\Sigma \ge \sqrt{\lambda_{\min}(\Sigma)} \, \norm A_{\rm F}$.

The Monge map $T$ between two Gaussians is the linear map $T(X)$ given in~\eqref{eq:gaussian_ot_map} and abusing notation slightly, we identify the map $T$ with the corresponding matrix, and we write $T(x) = Tx$. In particular, linearity of the Monge maps implies that the velocity vector field ${(v_t^\star)}$ associated to the Lagrangian coupling of the curve, is also linear for each $t$:  $v_t^\star$ is a symmetric linear mapping $\R^d\to\R^d$, that is, there exists a symmetric matrix $V_t^\star \in \R^{d\times d}$ such that $v_t^\star(x) = V_t^\star x$.


\subsection{Control of the Velocities}

We write $\delta_i := t_{i+1} - t_i$ and $\delta := \max_{i\in [N]} \delta_i$.
The first step is to prove a quantitative bound on how well the Monge map $T_i$ approximates ${\id} + \delta_i v_{t_{i-1}}$. We prove a more general approximation result which may be of independent interest.

\begin{thm}\label{thm:transport_approx}
Let $t,t+h \in [0,1]$, where $h \ne 0$.
Write $\delta := \abs h$ and assume $\delta \le c\sqrt{\lambda_{\min}(\Sigma_t)}/L$, for some constant $0 < c < 1$.
Let $T$ denote the Monge map from $\mus_t$ to $\mus_{t+h}$, and let $\bar T : \R^d\to\R^d$ be another linear mapping satisfying the following properties:
\begin{enumerate}
    \item $\bar T$ can be identified with a symmetric matrix.
    \item $\norm{\bar T \X_t - \X_t}_{L^2(\Pr)} \le c\sqrt{\lambda_{\min}(\Sigma_t)}$.
\end{enumerate}
Then,
\begin{align*}
    \norm{T\X_t - \bar T\X_t}_{L^2(\Pr)} \le \frac{1+2c}{1-c} \, \norm{\bar T \X_t - \X_{t+h}}_{L^2(\Pr)}.
\end{align*}
\end{thm}
\begin{proof}
    Let $e := \X_{t+h} - \bar T\X_t$.

    Consider the quadratic function $\varphi : \R^d\to\R$ defined by $\varphi(x) := \langle x, A x \rangle$, where $A := (T-\bar T)/\norm{T-\bar T}_{\Sigma_t}$.
    Note that $A$ is symmetric (since $T$ and $\bar T$ are).
    Then,
    \begin{align*}
        \E \varphi(T\X_t)
        &= \E \varphi(\X_{t+h})
        = \E \varphi(\bar T\X_t + e).
    \end{align*}
    Expanding this out,
    \begin{align*}
        0
        &= \E\langle (T+\bar T) \X_t + e, A \{(T-\bar T) \X_t - e\} \rangle \\
        &= \E\langle (T+\bar T) \X_t, A (T-\bar T) \X_t \rangle + \text{error}.
    \end{align*}
    We next bound the error term.
    First, note that by our assumption,
    \begin{align*}
        \norm{T-I_d}_{\Sigma_t}
        &= W_2(\mus_t, \mus_{t+h})
        \le L\delta
        \le c\sqrt{\lambda_{\min}}, \\
        \norm{\bar T-I_d}_{\Sigma_t}
        &\le c\sqrt{\lambda_{\min}},
    \end{align*}
    where we write $\lambda_{\min} = \lambda_{\min}(\Sigma_t)$.
    The error term is split into two further terms.
    For the first term,
    \begin{align*}
        \abs{\E\langle e, A(T-\bar T) \X_t \rangle}
        &\le \norm{e}_{L^2(\Pr)} \, \norm{A(T-\bar T)}_{\Sigma_t} \\
        &\le \norm{e}_{L^2(\Pr)} \, \norm A_{\rm F} \, \norm{T-\bar T}_{\Sigma_t} \\
        &\le \norm{e}_{L^2(\Pr)} \, \frac{1}{\sqrt{\lambda_{\min}}} \, (\norm{T-I_d}_{\Sigma_t} + \norm{\bar T - I_d}_{\Sigma_t}) \\
        & \le 2c \, \norm{e}_{L^2(\Pr)},
    \end{align*}
    where we used the fact that $\norm A_{\Sigma_t} \le 1$ implies that $\norm A_{\rm F} \le 1/\sqrt{\lambda_{\min}}$.
    The second term is bounded by
    \begin{align*}
        \abs{\E\langle (T+\bar T) \X_t + e, Ae \rangle}
        &\le \abs{\E \langle T\X_t + \X_{t+h} - 2\X_t, Ae \rangle} + 2\abs{\E\langle \X_t, Ae \rangle} \\
        &\le \{\norm A_{\rm F} \, (\norm{T-I_d}_{\Sigma_t} + \norm{\X_{t+h} - \X_t}_{L^2(\Pr)})  + 2\norm A_{\Sigma_t}\} \, \norm{e}_{L^2(\Pr)} \\
        &\le 2 \, (1+c) \, \norm{e}_{L^2(\Pr)},
    \end{align*}
    where we used
    \begin{align*}
        \norm{\X_{t+h}-\X_t}_{L^2(\Pr)}^2
        &= \E\Bigl[\Bigl\lVert \int_t^{t+h} \dot X^\star_s \, \D s\Bigr\rVert^2\Bigr]\le \delta \, \Bigl\lvert \int_t^{t+h} \norm{\dot X^\star_s}_{L^2(\Pr)}^2 \, \D s\Bigr\rvert
        \le L^2 \delta^2.
    \end{align*}
    Thus, we have
    \begin{align*}
        2\norm{T-\bar T}_{\Sigma_t}
        &= 2\E\langle \X_t, A(T-\bar T) \X_t \rangle \\
        &= -\E\langle (T+\bar T - 2I_d) \X_t, A(T-\bar T) \X_t \rangle + \text{error} \\
        &\le (\norm{T-I_d}_{\Sigma_t} + \norm{\bar T - I_d}_{\Sigma_t}) \, \norm A_{\rm F} \, \norm{T-\bar T}_{\Sigma_t}  + \text{error} \\
        &\le 2c \, \norm{T-\bar T}_{\Sigma_t} + (2+4c) \, \norm{e}_{L^2(\Pr)}
    \end{align*}
    which finally yields
    \begin{align*}
        \norm{T-\bar T}_{\Sigma_t}
        &\le \frac{1+2c}{1-c} \, \norm{e}_{L^2(\Pr)}
    \end{align*}
    as required.\placeqed
\end{proof}

\begin{cor}\label{cor:transport_approx}
Let $t,t+h \in [0,1]$, where $h \ne 0$, and write $\delta := \abs h$.
Let $k \in \{0,1,2\}$, and suppose $\delta$ is small enough so that
\begin{align*}
    \sum_{i=1}^k \frac{R_i \delta^i}{i!} \le c\sqrt{\lambda_{\min}(\Sigma_t)},
\end{align*}
where we set $R_i := \sup_{t\in [0,1]}{\norm{\partial^i \X}_{L^2(\Pr)}}$.
Then,
\begin{align*}
    \Bigl\lVert T\X_t - \sum_{i=0}^k \frac{h^i}{i!} \, {(\partial^i \X)}_t \Bigr\rVert_{L^2(\Pr)}
    &\le \frac{1+2c}{1-c} \, \frac{R_{k+1} \delta^{k+1}}{(k+1)!}.
\end{align*}
\end{cor}
\begin{proof}
    We apply Theorem~\ref{thm:transport_approx} with
    \begin{align*}
        \bar T \X_t
        &= \sum_{i=0}^k \frac{h^i}{i!} \, {(\partial^i \X)}_t.
    \end{align*}
    Using $\dotX_t = V_t^\star \X_t$, where $V_t^\star$ is symmetric, we obtain:
    \begin{align*}
        \dotX_t
        &= V_t^\star \X_t, \\
        \ddotX_t
        &= \dot V_t^\star \X_t + V_t^{\star 2} \X_t
        = (\dot V_t^\star + V_t^{\star 2}) \X_t, \\
        \dddotX_t
        &= (\ddot V_t^\star + 2\dot V_t^\star V_t^\star + V_t^\star \dot V_t^\star + V_t^{\star 3}) \X_t, \\
        &\vdots
    \end{align*}
    Observe that the $i$th derivative of $t \mapsto \X_t$ at $t$ is indeed a linear function of $\X_t$, but for $i\ge 3$ it is no longer given by a symmetric matrix, so it no longer satisfies the first assumption of Theorem~\ref{thm:approx_thm}; this is why we restrict ourselves to $k = 0, 1, 2$.
    
    For the third assumption of Theorem~\ref{thm:approx_thm}, note that
    \begin{align*}
        \norm{\bar T \X_t - \X_t}_{L^2(\Pr)}
        &= \Bigl\lVert \sum_{i=1}^k \frac{h^i}{i!} \, {(\partial^i \X)}_t \Bigr\rVert_{L^2(\Pr)} \le \sum_{i=1}^k \frac{\delta^i R_i}{i!}
        \le c\sqrt{\lambda_{\min}(\Sigma_t)},
    \end{align*}
    by our assumption on $\delta$.
    
    Finally, the error $e := \X_{t+h} - \bar T\X_t$ is controlled via Taylor's theorem:
    \begin{align*}
        \norm e_{L^2(\Pr)}
        &= \Bigl\lVert \X_{t+h} - \sum_{i=0}^k \frac{h^i}{i!} \, {(\partial^i \X)}_t \Bigr\rVert_{L^2(\Pr)} \\
        &= \Bigl\lVert \int_t^{t+h} \frac{{(\partial^{k+1} \X)}_s}{k!} \, {(s-t)}^k \, \D s\Bigr\rVert_{L^2(\Pr)} \\
        &\le \frac{R_{k+1} \delta^{k+1}}{(k+1)!}. \placeqedmm
    \end{align*}
\end{proof}

\begin{rmk}\label{rmk:sharp}
If we let $\delta \searrow 0$, we can also take $c\searrow 0$, obtaining
\begin{align*}
    \limsup_{\delta\searrow 0} \frac{1}{\delta^{k+1}} \Bigl\lVert T\X_t - \sum_{i=0}^k \frac{h^i}{i!} \, {(\partial^i \X)}_t \Bigr\rVert_{L^2(\Pr)}
    &\le \frac{R_{k+1}}{(k+1)!}.
\end{align*}
Comparing this to a Euclidean Taylor expansion, this is apparently sharp.
\end{rmk}

Corollary~\ref{cor:transport_approx} says that in order to prove our desired result $\dot \Y_{t_{i-1}} = \dotX_{t_{i-1}} + O(\delta)$, it suffices to show that $\dot \Y_{t_{i-1}} = (T_i \Y_{t_{i-1}} - \Y_{t_{i-1}})/\delta_i + O(\delta)$ (since the RHS of both expressions equals $V_{t_{i-1}}^\star \Y_{t_{i-1}} = V_{t_{i-1}}^\star \X_{t_{i-1}}$ up to $O(\delta)$). Since the latter statement involves only the process $(\Y_t)$, it is easier to prove.

However, there is still a major difficulty to overcome: $\dot \Y_{t_{i-1}}$ is the velocity of an interpolating cubic spline, which depends on all of the interpolated points $\Y_{t_0}, \Y_{t_1},\dotsc,\Y_{t_N}$. In Appendix~\ref{appendix:cubic_splines}, we show that the derivative of the cubic spline interpolation can be understood in terms of the linear system of equations involving the quantities
\begin{align*}
    \Delta_i
    &:= \frac{\Y_{t_{i+1}} - \Y_{t_i}}{\delta_{i+1}} - \frac{\Y_{t_i} - \Y_{t_{i-1}}}{\delta_i}, \qquad i\in [N-1].
\end{align*}
Therefore, we next control these quantities.

\begin{prop}\label{prop:Delta_control}
Assume $\delta \le \sqrt{\lambda_{\min}}/(2L)$.
    For each $i \in [N-1]$, it holds that
    \begin{align*}
        \bigl\lVert \frac{\Y_{t_{i+1}} - \Y_{t_i}}{\delta_{i+1}} - \frac{\Y_{t_i} - \Y_{t_{i-1}}}{\delta_i} \bigr\rVert_{L^2(\Pr)}
        \le \frac{25}{4} R\delta.
    \end{align*}
\end{prop}
\begin{proof}
    From Corollary~\ref{cor:transport_approx},
    \begin{align*}
        &\bigl\lVert \frac{\Y_{t_i} - \Y_{t_{i-1}}}{\delta_i} - V^\star_{t_{i-1}} \Y_{t_{i-1}} \bigr\rVert_{L^2(\Pr)} = \bigl\lVert \frac{T_i - I_d}{\delta_i} - V^\star_{t_{i-1}} \bigr\rVert_{\Sigma_{t_{i-1}}}
        \le 2R\delta_i,
    \end{align*}
    where we use the fact that $\Y_{t_{i-1}} \sim \mus_{t_{i-1}}$ and that $\Y_{t_i} = T_i \Y_{t_{i-1}}$.
    Similarly,
    \begin{align*}
        \bigl\lVert \frac{\Y_{t_{i+1}} - \Y_{t_i}}{\delta_{i+1}} - V^\star_{t_i} \Y_{t_i} \bigr\rVert_{L^2(\Pr)}
        &\le 2R \delta_{i+1}.
    \end{align*}
    Therefore,
    \begin{align*}
        \norm{\Delta_i}_{L^2(\Pr)}
        &\le 4R\delta + \norm{V_{t_i}^\star \Y_{t_i} - V_{t_{i-1}}^\star \Y_{t_{i-1}}}_{L^2(\Pr)}.
    \end{align*}
    Since $\Y_{t_i} = T_i \Y_{t_{i-1}}$, we replace $T_i$ by $I_d + \delta_i V_{t_{i-1}}^\star$.
    \begin{align*}
        &\norm{V_{t_i}^\star \Y_{t_i} - V_{t_{i-1}}^\star \Y_{t_{i-1}}}_{L^2(\Pr)} \\
        &\qquad{} \le \norm{V_{t_i}^\star (T_i - I_d - \delta_i V_{t_{i-1}}^\star) \Y_{t_{i-1}}}_{L^2(\Pr)} + \norm{V_{t_i}^\star (I_d + \delta_i V_{t_{i-1}}^\star) \Y_{t_{i-1}} - V^\star_{t_{i-1}} \Y_{t_{i-1}}}_{L^2(\Pr)}.
    \end{align*}
    We control the first term using Corollary~\ref{cor:transport_approx}:
    \begin{align*}
        \norm{V^\star_{t_i} (T_i - I_d - \delta_i V^\star_{t_{i-1}}) \Y_{t_{i-1}}}_{L^2(\Pr)} &\le \norm{V^\star_{t_i}}_{\rm F} \, \norm{(T_i - I_d - \delta_i V^\star_{t_{i-1}}) \Y_{t_{i-1}}}_{L^2(\Pr)} \\
        &\le \frac{L}{\sqrt{\lambda_{\min}}} \, \norm{T_i - I_d - \delta_i V^\star_{t_{i-1}}}_{\Sigma_{t_{i-1}}} \\
        &\le \frac{L}{\sqrt{\lambda_{\min}}} \cdot 2R \delta_i^2
        \le R \delta_i,
    \end{align*}
    where we used $\norm{V^\star_{t_i}}_{\rm F}
    \le \lambda_{\min}^{-1/2}\,\norm{V^\star_{t_i}}_{\Sigma_{t_i}}
    \le L\lambda_{\min}^{-1/2}$ by our Lipschitz assumption.
    Now for the second term.
    Introduce the random trajectory $(\X_t)$ sampled from the true curve $(\mus_t)$ with the Lagrangian coupling, and couple the process $(\Y_t)$ with $(\X_t)$ by setting $\Y_{t_{i-1}} = \X_{t_{i-1}}$.
    Thus,
    \begin{align*}
        &\norm{V^\star_{t_i} (I_d + \delta_i V^\star_{t_{i-1}}) \Y_{t_{i-1}} - V^\star_{t_{i-1}} \Y_{t_{i-1}}}_{L^2(\Pr)} \\
        &\qquad \le \norm{V^\star_{t_i} \X_{t_i} - V^\star_{t_{i-1}} \X_{t_{i-1}}}_{L^2(\Pr)}+ \norm{V^\star_{t_i} \{(I_d + \delta_i V^\star_{t_{i-1}}) \X_{t_{i-1}} - \X_{t_i}\}}_{L^2(\Pr)}.
    \end{align*}
    It is easy to control
    \begin{equation*}
        \norm{V^\star_{t_i} \X_{t_i} - V^\star_{t_{i-1}} \X_{t_{i-1}}}_{L^2(\Pr)} = \Bigl\lVert \int_{t_{i-1}}^{t_i} \ddotX_t \, \D t \Bigr\rVert_{L^2(\Pr)} \le R \delta_i.
    \end{equation*}
    Lastly,
    \begin{align*}
        \norm{V^\star_{t_i} \{(I_d + \delta_i V^\star_{t_{i-1}}) \X_{t_{i-1}} - \X_{t_i}\}}_{L^2(\Pr)} &\le \norm{V^\star_{t_i}}_{\rm F} \, \norm{\X_{t_i} - \X_{t_{i-1}} - \delta_i V^\star_{t_{i-1}} \X_{t_{i-1}}}_{L^2(\Pr)} \\
        &\le \frac{L}{\sqrt{\lambda_{\min}}} \, \Bigl\lVert \int_{t_{i-1}}^{t_i} \int_{t_{i-1}}^t \ddotX_s \, \D s \, \D t \Bigr\rVert_{L^2(\Pr)} \\
        & \le \frac{L}{\sqrt{\lambda_{\min}}} \cdot \frac{R \delta_i^2}{2}
        \le \frac{R\delta_i}{4}.
    \end{align*}
    Putting it all together, we obtain
    \begin{align*}
        \norm{\Delta_i}_{L^2(\Pr)}
        &\le \frac{25}{4} R\delta. \placeqedmm
    \end{align*}
\end{proof}

To match notation with Appendix~\ref{appendix:cubic_splines}, we set
\begin{align*}
    M_i
    &:= \ddot \Y_{t_{i-1}}, \qquad i \in [N+1].
\end{align*}

\begin{lem}\label{lem:M_control}
Assume $\delta \le \sqrt{\lambda_{\min}}/(2L)$.
It holds that
\begin{align*}
    \norm{M_i}_{L^2(\Pr)}
    &\le \frac{75 {(1+\alpha)}^2}{4\alpha^3} \,R.
\end{align*}
\end{lem}
\begin{proof}
    As described in Appendix~\ref{appendix:cubic_splines}, we know that $M = 6\mb T^{-1} \Delta$, where the entries of $\mb T^{-1}$ are bounded in Lemma~\ref{lem:T_inv}.
    Thus,
    \begin{align*}
        \norm{M_i}_{L^2(\Pr)}
        &= 6\,\Bigl\lVert \sum_{j=1}^{N-1} {(\mb T^{-1})}_{i,j} \Delta_j \Bigr\rVert_{L^2(\Pr)} \\
        &\le 6\sum_{j=1}^{N-1} \abs{{(\mb T^{-1})}_{i,j}}\, \norm{\Delta_j}_{L^2(\Pr)} \\
        &\le 6\sum_{j=1}^{N-1} \frac{1}{4\alpha^2 \delta} \, \frac{1}{{(1+\alpha)}^{\abs{i-j}-1}} \, \frac{25}{4} R\delta \\
        &\le \frac{75R}{4\alpha^2} \sum_{k=0}^\infty \frac{1}{{(1+\alpha)}^{k-1}}
        = \frac{75 {(1+\alpha)}^2}{4\alpha^3} \,R,
    \end{align*}
    where we use Proposition~\ref{prop:Delta_control}. \placeqed
\end{proof}

Finally, we are ready to state our control on the velocity of the trajectory ${(\Y_t)}$.

\begin{prop}\label{prop:velocity_control}
Assume $\delta \le \sqrt{\lambda_{\min}}/(2L)$.
Then,
\begin{align*}
    \norm{\dot \Y_{t_{i-1}} - \dotX_{t_{i-1}}}_{L^2(\Pr)}
    &\le \frac{16\alpha^3 + 75 {(1+\alpha)}^2}{8\alpha^3} \, R\delta.
\end{align*}
\end{prop}
\begin{proof}
    It holds that
    \begin{align*}
        \dot \Y_{t_{i-1}}
        &= \frac{\Y_{t_i} - \Y_{t_{i-1}}}{\delta_i} - \frac{M_{i+1}+2M_i}{6} \, \delta_i
    \end{align*}
    (see Appendix~\ref{appendix:cubic_splines}).
    Therefore,
    \begin{align*}
        \bigl\lVert \dot \Y_{t_{i-1}} - \frac{\Y_{t_i} - \Y_{t_{i-1}}}{\delta_i} \bigr\rVert_{L^2(\Pr)}  & \le \frac{\norm{M_{i+1}}_{L^2(\Pr)} + 2\norm{M_i}_{L^2(\Pr)}}{6} \, \delta \le \frac{75 {(1+\alpha)}^2}{8\alpha^3} \,R\delta,
    \end{align*}
    by Lemma~\ref{lem:M_control}.
    Next, we recall that $\Y_{t_i} = T_i \Y_{t_{i-1}}$, and that $(\Y_t)$ and $(\X_t)$ are coupled so that $\Y_{t_{i-1}} = \X_{t_{i-1}}$.
    Thus,
    \begin{align*}
        \norm{\dot \Y_{t_{i-1}} - \dotX_{t_{i-1}}}_{L^2(\Pr)} &\le \bigl\lVert \dot \Y_{t_{i-1}} - \frac{T_i \Y_{t_{i-1}} - \Y_{t_{i-1}}}{\delta_i} \bigr\rVert_{L^2(\Pr)} + \bigl\lVert \dotX_{t_{i-1}} - \frac{T_i \X_{t_{i-1}} - \X_{t_{i-1}}}{\delta_i} \bigr\rVert_{L^2(\Pr)} \\
        &\le \frac{75 {(1+\alpha)}^2}{8\alpha^3} \,R\delta + 2R\delta,
    \end{align*}
    where we invoke Corollary~\ref{cor:transport_approx} again. \placeqed
\end{proof}

\subsection{Curvature of the Transport Spline}

Next, we must bound the curvature of $(\Y_t)$, but this is an easy task given what we have established so far.

\begin{prop}\label{prop:curv_transport_spline}
Assume $\delta \le \sqrt{\lambda_{\min}}/(2L)$.
Then,
\begin{align*}
    \sup_{t\in [0,1]}{\norm{\ddot \Y_t}_{L^2(\Pr)}}
    &\le\frac{75 {(1+\alpha)}^2}{4\alpha^3} \,R.
\end{align*}
\end{prop}
\begin{proof}
    Indeed, $t\mapsto \ddot \Y_t$ is a piecewise linear function (see Appendix~\ref{appendix:cubic_splines}), so it is maximized at the knots.
    For $t\in [t_{i-1}, t_i]$, it follows that
    \begin{align*}
        \norm{\ddot \Y_t}_{L^2(\Pr)}
        &= \bigl\lVert \frac{t_i - t}{\delta_i} \, \ddot \Y_{t_{i-1}} + \frac{t-t_{i-1}}{\delta_i} \, \ddot \Y_{t_i} \bigr\rVert_{L^2(\Pr)} \\
        &\le \frac{t_i - t}{\delta_i} \, \norm{\ddot \Y_{t_{i-1}}}_{L^2(\Pr)} + \frac{t-t_{i-1}}{\delta_i} \, \norm{\ddot \Y_{t_i}}_{L^2(\Pr)} \\
        &\le \norm{\ddot \Y_{t_{i-1}}}_{L^2(\Pr)} \vee \norm{\ddot \Y_{t_i}}_{L^2(\Pr)} \\
        &= \norm{M_i}_{L^2(\Pr)} \vee \norm{M_{i+1}}_{L^2(\Pr)} \\
        &\le \frac{75 {(1+\alpha)}^2}{4\alpha^3} \,R,
    \end{align*}
    by Lemma~\ref{lem:M_control}. \placeqed
\end{proof}

\subsection{Proof of the Main Theorem}\label{appendix:finally_the_proof}

\begin{proof}[Proof of Theorem~\ref{thm:approx_thm}]
    Let $t \in [t_{i-1}, t_i]$, and let the processes $(\Y_t)$ and $(\X_t)$ be coupled with $\Y_{t_{i-1}} = \X_{t_{i-1}}$.
    Then,
    \begin{align*}
        \norm{\Y_t - \X_t}_{L^2(\Pr)}
        &\le \delta_i \, \norm{\dot \Y_{t_{i-1}} - \dotX_{t_{i-1}}}_{L^2(\Pr)} + \Bigl\lVert \int_{t_{i-1}}^{t_i} \int_{t_{i-1}}^t (\ddot \Y_s - \ddotX_s) \, \D s \, \D t \Bigr\rVert_{L^2(\Pr)} \\
        &\le \frac{16\alpha^3 + 75 {(1+\alpha)}^2}{8\alpha^3} \, R\delta^2 + \frac{\delta^2}{2} \sup_{t\in [0,1]}{(\norm{\ddot \Y_t}_{L^2(\Pr)} +\norm{\ddotX_t}_{L^2(\Pr)})} \\
        &\le \frac{10\alpha^3 + 75 {(1+\alpha)}^2}{4\alpha^3} \, R\delta^2
        \le \frac{115}{2\alpha^3} \, R\delta^2,
    \end{align*}
    where we have used Proposition~\ref{prop:velocity_control} and Proposition~\ref{prop:curv_transport_spline}. \placeqed
\end{proof}

\subsection{Piecewise Geodesic Interpolation}\label{appendix:piecewise_geodesic}

In this section, we study the approximation error of piecewise geodesic interpolation. Namely, we define a stochastic process, still denoted $(\Y_t)$, as follows.
\begin{enumerate}
    \item Draw $\Y_{t_0} \sim \mu_{t_0}$.
    \item For $i=1,\dotsc,N$, set $\Y_{t_i} := T_i(\Y_{t_{i-1}})$.
    \item We join the points $\Y_{t_0}, \Y_{t_1},\dotsc,\Y_{t_N}$ via straight lines. Namely, for $t \in [t_{i-1}, t_i]$ we set
    \begin{align*}
        \Y_t
        &= \frac{t_i - t}{t_i - t_{i-1}} \, \Y_{t_{i-1}} + \frac{t-t_{i-1}}{t_i - t_{i-1}} \, \Y_{t_i}.
    \end{align*}
\end{enumerate}
Let $\mu_t$ denote the law of $\Y_t$.

\begin{thm}\label{thm:piecewise_geodesic_approx}
Let the notation and assumptions of Theorem~\ref{thm:approx_thm} hold (except for the definition of $(\mu_t)$).
Then,
\begin{align*}
    \sup_{t\in [0,1]} W_2(\mu_t,\mus_t)
    &\le \frac{5}{2} R\delta^2.
\end{align*}
\end{thm}
\begin{proof}
    As in Appendix~\ref{appendix:finally_the_proof}, we have
    \begin{align*}
        \norm{\Y_t - \X_t}_{L^2(\Pr)}
        &\le \delta_i \, \norm{\dot \Y_{t_{i-1}} - \dotX_{t_{i-1}^{\tiny +}}}_{L^2(\Pr)}  + \Bigl\lVert \int_{t_{i-1}}^{t_i} \int_{t_{i-1}}^t \ddotX_s \, \D s \, \D t \Bigr\rVert_{L^2(\Pr)} \\
        &\le 2R\delta^2 + \frac{1}{2} R\delta^2.
    \end{align*}
    Here, we use several facts: (1) $\dot \Y_{t_{i-1}^{\tiny +}}$, the derivative of $(\Y_t)_t$ at $t_{i-1}$ from the right, equals \[ (T_i - I_d) \Y_{t_{i-1}} = (T_i - I_d) \X_{t_{i-1}}, \] and so we can apply Corollary~\ref{cor:transport_approx}; (2) the curve $(\Y_t)$, consisting of piecewise straight lines, has no acceleration.
    This finishes the proof. \placeqed
\end{proof}

Formally, Theorem~\ref{thm:piecewise_geodesic_approx} is a slightly better approximation guarantee than Theorem~\ref{thm:approx_thm}. Theorem~\ref{thm:piecewise_geodesic_approx} can also be strengthened asymptotically to
\begin{align*}
    \limsup_{\delta\searrow 0} \frac{1}{\delta^2} \sup_{t\in [0,1]} W_2(\mu_{\delta,t}, \mus_t) \le R,
\end{align*}
as in Section~\ref{scn:approx_guarantee}.
Of course, we do not advocate for using piecewise geodesic interpolation because it is unsuitable for trajectory estimation (see Figure~\ref{fig:sawtooth}).

\section{\allcaps{Natural Cubic Splines}}\label{appendix:cubic_splines}

For the reader's convenience and to make the paper more self-contained, in this section we present a derivation of natural cubic splines and some of their properties.
The results obtained here are used in Appendix~\ref{appendix:approx} for the proof of the main approximation result (Theorem~\ref{thm:approx_thm}).

We are given times $0 = t_0 < t_1 < \cdots < t_N = 1$ and corresponding points $(x_{t_0}, x_{t_1},\dotsc,x_{t_N})$ in $\R^d$.
Our goal is to construct a piecewise cubic polynomial interpolation $y : [0,1] \to \R^d$ which is $\mc C^2$ smooth.

We parametrize $y$ in the following way: for each $i \in [N]$ and for $t \in [t_{i-1}, t_i]$, we set $y(t) = y_i(t)$, where
\begin{align*}
    y_i(t)
    = a_i \, {(t-t_{i-1})}^3 &+ b_i \, {(t-t_{i-1})}^2  + c_i \, (t-t_{i-1}) + d_i.
\end{align*}

Computing derivatives,
\begin{align*}
    x_{t_{i-1}}
    = y_i(t_{i-1})
    &= d_i, \\
    x_{t_i}
    = y_i(t_i)
    &= a_i \delta_i^3 + \frac{m_i}{2} \delta_i^2 + c_i \delta_i + d_i, \\
    \dot y_i(t_{i-1})
    &= c_i, \\
    \dot y_i(t_i)
    &= 3a_i \delta_i^2 + m_i \delta_i + c_i, \\
    \ddot y_i(t_{i-1})
    &= m_i, \\
    \ddot y_i(t_i)
    &= 6a_i \delta_i + m_i,
\end{align*}
where define $\delta_i := t_i - t_{i-1}$ and $m_i := 2b_i$ (and anticipating the natural boundary condition, which asserts $\ddot y(0) = \ddot y(1) = 0$, we make the convention $m_{N+1} := 0$).
Using continuity of the first and second derivatives of $y$ at the knots, we solve for the coefficients of the polynomial $y_i$ in terms of the variables $m$ and $x$:
\begin{align*}
    a_i
    &= \frac{m_{i+1} - m_i}{6\delta_i}, \\
    b_i
    &= \frac{m_i}{2}, \\
    c_i
    &= \frac{x_{t_i} - x_{t_{i-1}}}{\delta_i} - \frac{m_{i+1} + 2m_i}{6} \, \delta_i, \\
    d_i
    &= x_{t_{i-1}}.
\end{align*}
Therefore, it suffices to work with the variables $m$.

If we plug these equations back into the continuity condition for the first derivative at the knot, after some algebra we obtain the equations
\begin{align*}
    6\Delta_i
    &= \delta_i m_i + 2(\delta_i + \delta_{i+1}) m_{i+1} + \delta_{i+2} m_{i+2}, \qquad i = 1,\dotsc, N-1,
\end{align*}
where we have defined the quantities
\begin{align*}
    \Delta_i
    &:= \frac{x_{t_{i+1}} - x_{t_i}}{\delta_{i+1}} - \frac{x_{t_i} - x_{t_{i-1}}}{\delta_i},
\end{align*}
a proxy for the second derivative of the data points.

We can express these equations in matrix form (including also the natural boundary condition $m_1 = 0$):
\begin{align*}
    \underbrace{\begin{bmatrix}
    2(\delta_1 + \delta_2) & \delta_2 && \\
    \delta_2 & \ddots & \ddots & \\
    & \ddots & \ddots & \delta_{N-1} \\
    && \delta_{N-1} & 2(\delta_{N-1} + \delta_N)
    \end{bmatrix}}_{:= \mb T} m
    = 6\Delta.
\end{align*}

The matrix $\mb T$ above is a symmetric tridiagonal matrix of size $N-1$.\footnote{To be precise, we should write this as the block matrix equation $(\mb T\otimes I_d) m = 6\Delta$.}
To obtain bounds on $m$, we will study the inverse $\mb T^{-1}$ of $\mb T$.

\begin{lem}\label{lem:T_inv}
Assume that for each $i\in [N]$, we have
$\alpha \delta \le t_i - t_{i-1} \le \delta.$
Then, we have the entrywise bound
\begin{align*}
    \abs{{(\mb T^{-1})}_{i,j}}
    &\le \frac{1}{4\alpha^2 \,{(1+\alpha)}^{\abs{i-j}-1}} \, \frac{1}{\delta}, \qquad i,j \in [N-1].
\end{align*}
\end{lem}
\begin{proof}
    We write $\mb T = \mb B + \mb D$, where
    \begin{align*}
        \mb B
        &:= \begin{bmatrix}
    0 & \delta_2 && \\
    \delta_2 & \ddots & \ddots & \\
    & \ddots & \ddots & \delta_{N-1} \\
    && \delta_{N-1} & 0
    \end{bmatrix}, \\
            \mb D
        &:= 2\diag(\delta_1+\delta_2,\dotsc,\delta_{N-1} + \delta_N).
    \end{align*}
    Therefore,
    \begin{align*}
        \mb T^{-1}
        &= {(\mb B + \mb D)}^{-1} \\
        &= \mb D^{-1/2} {(I_{N-1} + \mb D^{-1/2} \mb B \mb D^{-1/2})}^{-1} \mb D^{-1/2} \\
        &= \sum_{k=0}^\infty {(-1)}^k \mb D^{-1/2} {(\underbrace{\mb D^{-1/2} \mb B \mb D^{-1/2}}_{:= \mb M})}^k \mb D^{-1/2}.
    \end{align*}
    The matrix $\mb M$ is
    \begin{align*}
        \mb M
        &= \begin{bmatrix}
    0 & \gamma_2 && \\
    \gamma_2 & \ddots & \ddots & \\
    & \ddots & \ddots & \gamma_{N-1} \\
    && \gamma_{N-1} & 0
    \end{bmatrix},
    \end{align*}
    where we set
    \begin{align*}
        \gamma_i := \frac{\delta_i}{2\sqrt{(\delta_{i-1} + \delta_i)(\delta_i + \delta_{i+1})}}
        \le \frac{1}{2(1+\alpha)}.
    \end{align*}
    Since $\mb M$ has non-negative entries, we have the entrywise bound
    \begin{align*}
        \mb M^k
        &\le \frac{1}{{\{2(1+\alpha)\}}^k} {\underbrace{\begin{bmatrix}
    0 & 1 && \\
    1 & \ddots & \ddots & \\
    & \ddots & \ddots & 1 \\
    && 1 & 0
    \end{bmatrix}}_{:=\mb A}}^k.
    \end{align*}
    The matrix $\mb A$ is the adjacency matrix of the path graph on $\{1,\dotsc,N-1\}$, so ${(\mb A^k)}_{i,j}$ is the number of paths from $i$ to $j$ of length $k$.
    We can trivially bound this number by $2^k \one_{\abs{i-j} \le k}$. From this we deduce the entrywise bound
    \begin{align*}
        {(\mb M^k)}_{i,j}
        &\le \frac{1}{{(1+\alpha)}^k} \one_{\abs{i-j} \le k}.
    \end{align*}
    Therefore,
    \begin{align*}
        \abs{{(\mb T^{-1})}_{i,j}}
        &\le \sum_{k=0}^\infty \frac{1}{2\sqrt{(\delta_i + \delta_{i+1})(\delta_j + \delta_{j+1})}} \, \frac{\one_{\abs{i-j} \le k}}{{(1+\alpha)}^k} \\
        &\le \frac{1}{4\alpha \delta} \sum_{k=\abs{i-j}}^\infty \frac{1}{{(1+\alpha)}^k} = \frac{1}{4\alpha^2 \,{(1+\alpha)}^{\abs{i-j}-1}} \, \frac{1}{\delta}. \placeqedmm
    \end{align*}
\end{proof}

\section{\allcaps{Details for the Experiments}}

In this section we provide further details for the experiments in the paper, except for the thin-plate spline example (which is discussed in Appendix~\ref{appendix:thin_plate}).

\subsection{Figure~\ref{fig:sawtooth}}\label{appendix:sawtooth}
In this figure, we set five Gaussians as our interpolation knots, alternating between
\begin{align*}
    \Normal\Bigl(\begin{bmatrix} 7(k-1) \\ 0 \end{bmatrix}, \begin{bmatrix}4 & 0\\ 0 & 2\end{bmatrix}\Bigr) \qquad\text{for}~k~\text{odd}
\end{align*}
and
\begin{align*}
    \Normal\Bigl(\begin{bmatrix} 7(k-1) \\ 7 \end{bmatrix}, \begin{bmatrix} 2 & 0\\ 0 & 4\end{bmatrix}\Bigr) \qquad\text{for}~k~\text{even},
\end{align*}
where $k = 1, \dotsc, 5$.

To determine the linear and cubic spline interpolations we first computed the optimal transport maps between the neighboring Gaussians. The closed-form formula for the Monge map from $\Normal(\mu_1, \Sigma_1)$ to $\Normal(\mu_2, \Sigma_2)$ is
\begin{align*}
    T(x) = \mu_2 + A(x - \mu_1), \quad A = \Sigma_1^{-\frac{1}{2}}(\Sigma_1^\frac{1}{2}\Sigma_2\Sigma_1^\frac{1}{2})^\frac{1}{2}\Sigma_1^{-\frac{1}{2}}.
\end{align*}
The gray lines in both figures show the trajectories of individual sample points along our interpolations. To draw them, we obtained a sample $X_0$ from the Gaussian at time $t = 0$, repeatedly applied the Monge maps between successive Gaussians in time, and fit a piecewise linear or natural cubic spline through these points as described in Section~\ref{scn:transport_spline_alg}.

Since the maps between successive Gaussians are linear and the formula for the linear or natural cubic spline is linear in its knots, the value of the spline $S_t(X_0)$ interpolation at time $t$ is linear in $X_0$. Hence, given the covariance of the Gaussian at time $t = 0$, we used this linear map $S_t$ to compute the covariance of the interpolated Gaussian at time $t$. Likewise, by taking a linear or cubic spline through the means of the Gaussians at the knot points, we obtained the means of the interpolated Gaussians at any given time. Using this information, we plotted the interpolated Gaussians at the halfway points between the knots for both the linear and cubic spline interpolations.

\subsection{Figure~\ref{fig:nbody}}\label{appendix:nbody}
To simulate the $n$-body trajectories, we used the Python \texttt{nBody} simulator by Cabrera \& Li, which can be accessed at \href{https://github.com/GabrielSCabrera/nBody}{https://github.com/GabrielSCabrera/nBody}.

We created 15 smaller bodies, each of mass $5\times 10^9$ and radius $1$.
Each body was initialized with a position $x$ and a velocity $v$ drawn randomly according to
\begin{align*}
    x \sim \Normal\Bigl(\begin{bmatrix} 100 \\ 100 \end{bmatrix}, \begin{bmatrix} 30 & 0 \\ 0 & 20 \end{bmatrix}\Bigr), \qquad
    v \sim \Normal\Bigl(\begin{bmatrix} 10 \\ -20 \end{bmatrix}, \begin{bmatrix} 20 & 0 \\ 0 & 10 \end{bmatrix}\Bigr).
\end{align*}
In addition, we also created one larger body, with mass $10^{11}$ and radius $10$, initialized at the origin with no initial velocity.

We simulated the trajectories of the planets for 5 seconds sampled every 0.02 seconds. We took the positions of the bodies at 5 evenly spaced times as the knots for our interpolation. In order to solve the matching problem between planets at neighboring knot times, we placed a uniform empirical distribution over the planets at both times and used the Python Optimal Transport (POT) library function \texttt{ot.emd} to compute the Monge map between these two distributions. We checked post process that the Monge maps computed were indeed valid matchings (i.e.\ permutation matrices).

Given the Monge maps between knots, we applied Algorithm~\ref{alg:interpolate} to interpolate the empirical distributions of the bodies using cubic splines. Note that in our cubic spline reconstruction, it is possible to observe mistakes in the matching, i.e., the Monge map may not necessarily map a body at one time to the same body at a future time. Such mismatches seem unavoidable without using a more sophisticated method which takes into account the physical model in the simulation.

\section{\allcaps{Further Details for Thin-Plate Splines}}\label{appendix:thin_plate}

\subsection{Simultaneously Optimal Coupling}\label{appendix:1d_coupling}

In Section~\ref{scn:thin_plate} we introduce the following coupling.
Let $U$ be a uniform random variable on $[0,1]$, and set
\begin{align*}
    \Y_{x_i}
    &= F^{-1}_{\mus_{x_i}}(U), \qquad i=0,1,\dotsc,N.
\end{align*}
Then, $(\Y_{x_0},\Y_{x_1},\dotsc,\Y_{x_N})$ is a simultaneously optimal coupling of the measures $\mus_{x_0}, \mus_{x_1},\dotsc,\mus_{x_N}$. This follows directly from~\textcite[\S 2.1-2.2]{santambrogio2015ot}, but we provide some additional explanation here.

As described in Section~\ref{scn:ot}, the Monge map $T_{i,j}$ from $\mus_{x_i}$ to $\mus_{x_j}$ is characterized as the ($\mus_{x_i}$-a.e.) unique mapping which both pushes $\mus_{x_i}$ forward to $\mus_{x_j}$ and is the gradient of a convex function.
In one dimension, the latter condition simply means that $T_{i,j}$ is an increasing function.
It is easily checked that $F_{\mus_{x_j}}^{-1} \circ F_{\mus_{x_i}}$ satisfies these properties, and thus\footnote{The inverse CDFs described here exist because of our assumption of absolute continuity of the measures.}
\begin{align*}
    T_{i,j} = F_{\mus_{x_j}}^{-1} \circ F_{\mus_{x_i}}.
\end{align*}

Now, observe that a composition of increasing maps is increasing, which implies that $T_{j,k} \circ T_{i,j}$ must be the Monge map $T_{i,k}$. This key fact directly implies the existence of the simultaneously optimal coupling of the measures. In higher dimensions, this breaks down because the composition of Monge maps is no longer necessarily a Monge map (that is, the composition of gradients of functions is not necessarily the gradient of a function).

\subsection{Gaussian Splines and Quantiles}\label{appendix:quantile}

Recall that the \emph{$\alpha$-quantile} of a measure $\mu$ is the value $c_\alpha$ for which $\mu((-\infty, c_\alpha]) = \alpha$. If $\mu$ has CDF $F_\mu$, then the $\alpha$-quantile is simply $F_\mu^{-1}(\alpha)$.
If we denote by $\Phi$ the CDF of the standard Gaussian distribution, then the $\alpha$-quantile of $\Normal(0, 1)$ is $\Phi^{-1}(\alpha)$, and the $\alpha$-quantile of $\mc N(m, \sigma^2)$ is $m+\Phi^{-1}(\alpha) \sigma$.
 
Suppose the measures $\mus_{x_i}$, $i=0,1,\dotsc,N$, are all one-dimensional Gaussians, and write $\mus_{x_i} = \Normal(m_{x_i}, \sigma_{x_i}^2)$.
The next result immensely facilitates the computation of the quantiles of the thin-plate transport spline.

\begin{prop}
Consider:
\begin{itemize}
    \item ${(m_x)}_{x\in\R^2}$, the (Euclidean) thin-plate spline interpolating the means ${(m_{x_i})}_{i=0}^N$, and
    \item ${(s_x)}_{x\in\R^2}$, the (Euclidean) thin-plate spline interpolating the standard deviations ${(\sigma_{x_i})}_{i=0}^N$.
\end{itemize}

For any $\alpha \in [0,1]$, the $\alpha$-quantile of $\mu_x$, the interpolated thin-plate transport spline at $x$, is given by $m_x + \Phi^{-1}(\alpha) \, \abs{s_x}$.
\end{prop}
\begin{proof}
    It is standard that there is a linear mapping $S_x$ such that the Euclidean thin-plate spline interpolating through ${(x_i, z_i)}_{i=0}^N$ is given by $S_x(z_0, z_1,\dotsc,z_N)$.
    
    It follows from~\eqref{eq:gaussian_ot_map} and the discussion in Appendix~\ref{appendix:1d_coupling} that the Monge map from $\mus_{x_0}$ to $\mus_{x_i}$ is the increasing map $z \mapsto (\sigma_{x_i}/\sigma_{x_0}) (z - m_{x_0}) + m_{x_i}$.
    Thus,
    \begin{align*}
        \Y_x
        &= S_x(\Y_{x_0}, \Y_{x_1}, \dotsc,\Y_{x_N}) \\
        &= S_x\bigl(\Y_{x_0}, \frac{\sigma_{x_1}}{\sigma_{x_0}} (\Y_{x_0} - m_{x_0}) + m_{x_1}, \dotsc,  \frac{\sigma_{x_N}}{\sigma_{x_0}} (\Y_{x_0} - m_{x_0}) + m_{x_N}\bigr) \\
        &= S_x(m_{x_0}, m_{x_1}, \dotsc, m_{x_N}) + S_x\bigl( \Y_{x_0} - m_{x_0}, \frac{\sigma_{x_1}}{\sigma_{x_0}} (\Y_{x_0} - m_{x_0}), \dotsc, \frac{\sigma_{x_N}}{\sigma_{x_0}} (\Y_{x_0} - m_{x_0})\bigr) \\
        &= m_x + \frac{\Y_{x_0} - m_{x_0}}{\sigma_{x_0}} \, S_x(\sigma_{x_0}, \sigma_{x_1},\dotsc,\sigma_{x_N}) \\
        &= m_x + s_x \, \frac{\Y_{x_0} - m_{x_0}}{\sigma_{x_0}}
        \sim \mc N(m_x, s_x^2)
        = \mu_x.
    \end{align*}
    This is the desired result. \placeqed
\end{proof}

\subsection{Figure~\ref{fig:caltemp}}\label{appendix:caltemp}

Here we give more details on the thin-plate spline interpolation leading to Figure~\ref{fig:caltemp}. The data is a representation of the temperature at various weather stations throughout California on June 1 of each year in a thirty year period. That is, we consider the {\it distribution} of temperatures recorded on each of June 1, 1981, June 1, 1982, \ldots, June 1, 2010, and we model this distribution as Gaussian  (characterized by its mean and standard deviation). This data is processed and released each decade by the NOAA NCEI \parencite{NOAA_normals_2010}. We interpolate these measures using our transport spline technique, obtaining Gaussian measures at each point in California. The left side of Figure~\ref{fig:caltemp} summarizes these measures by their quantiles, while the right side illustrates the behavior of our method as we sample increasingly many weather stations. The median temperature in the top left quantile plot is taken to be equal to the mean temperature due to our assumption that the temperature distribution is Gaussian at every location. Though there are 484 stations in the NOAA dataset, we used substantially fewer to better capture the convergence of our method.

\printbibliography{}

\end{document}